%% file: arxiv-02.tex
\documentclass[preprint,3p]{elsarticle}

\input{preambula.tex}

\journal{arXiv}

\begin{document}

\begin{frontmatter}
\title{On P-like ideals induced by disjoint families} 
\author{{Adam Marton}\fnref{fn1}}
\ead{adam.marton@student.upjs.sk}
\author{{Jaroslav \v Supina}\fnref{fn2}}
\address{Institute of Mathematics,
Pavol Jozef \v{S}af\'arik University in Ko\v sice,
Jesenn\'a 5, 040 01 Ko\v{s}ice, Slovakia}
\ead{jaroslav.supina@upjs.sk}
\fntext[fn1]{The~first author was supported by the~Slovak Research and Development Agency under the~Contract no. APVV-20-0045.}
\fntext[fn2]{The~second
author  was supported by the~Slovak Research and Development Agency under the~Contract no. APVV-20-0045 and by the~grant 1/0657/22 of Slovak Grant Agency VEGA.}
\begin{abstract}
We consider a~combinatorial property isolated in the field of ideal convergence, a~P-property for two ideals on natural numbers. We show that among selected ideals induced by disjoint families, not all pairs satisfy P-property for two ideals. In many cases we specify the inducing partitions for which the corresponding ideals possess the property. Regarding selector ideals, a~useful coloring-like necessary condition is provided. 
\end{abstract}
\begin{keyword}
ideal \sep P-ideal \sep disjoint families
\end{keyword}
\end{frontmatter}

\section{Introduction}

\par

We study a~combinatorial property of a~pair $\I,\J$ of ideals on~$\omega$. Namely, we consider the~following variation of a~P-ideal:
\begin{defi}\label{definiciaPJ}
We say that an~ideal $\I$ is a~$\mathrm{P}(\J)$-ideal if for any family $\{I_n:n\in\omega \}\subseteq\I$ there is $I\in\I$ such that $I_n\subseteq^\J I$ for each $n\in\omega$.
\end{defi}
The~property was introduced by M.~Mačaj and M.~Sleziak~\cite{macslez}, and applied later in~\cite{filipow, stanfil, stan, rep1, rep2}. In addition to its equivalent formulation via the Boolean algebra $\mathcal{P}(\omega)/\J$, M.~Mačaj and M.~Sleziak found the~property to be a~useful tool in characterizing certain properties of ideal convergence we describe below. They define a~sequence $\seqn{x_n}{n}$ in a~topological space~$X$ to be $\I^\J$-convergent to a~point $x\in X$ if there is $F\in\I^\ast$ such that the~sequence $\gseqn{x_n}{n\in F}$ is $(\J\restriction F)$-convergent to~$x$. They have shown
\begin{theorem*}[M. Mačaj -- M. Sleziak]
Let $X$ be a~non-discrete {\rm T}$_1$ first countable topological space. The~following are equivalent:
\begin{enumerate}[(1)]
    \item $\I$ is a~$\mathrm{P}(\J)$-ideal.
    \item The family of all equivalence classes corresponding to sets in $\I$ is  $\sigma$-directed in the quotient $\mathcal{P}(\omega)/\J$.
    \item For any sequence $\seqn{x_n}{n}$ in~$X$, if $\seqn{x_n}{n}$ is $\I$-convergent to~$x$ then $\seqn{x_n}{n}$ is $\I^\J$-convergent to~$x$.
\end{enumerate}
\end{theorem*}
In~\cite{filipow} R.~Filipów and M.~Staniszewski study an ideal version of equal convergence~\cite{CsaLa}, called quasi-normal convergence in~\cite{Buk91}. In fact, various versions of ideal equal convergence have been introduced in \cite{das, szuca} but the most general one has appeared in this very paper \cite{filipow}. A~sequence $\seqn{f_n}{n}$ of real-valued functions defined on a~set~$X$ is $(\I,\J)$-equally convergent to a~function $f$ if there is a sequence of positive reals $\seqn{\varepsilon_n}{n}$ $\J$-converging to $0$, such that $\set{n}{|f_n(x)-f(x)|\geq\varepsilon_n}\in\I$ for every $x\in X$. They have shown
\begin{theorem*}[R.~Filipów -- M.~Staniszewski]
Let $X$ be a~non-empty set. The~following are equivalent:
\begin{enumerate}[(1)]
    \item $\I$ is a~$\mathrm{P}(\J)$-ideal.
   \item For any sequence $\seqn{f_n}{n}$ of real-valued functions on~$X$, if $\seqn{f_n}{n}$ is $(\I,\I)$-equally convergent to~$f$ then $\seqn{f_n}{n}$ is $(\I,\J)$-equally convergent to~$f$.
    \item For any sequence $\seqn{f_n}{n}$ of real-valued functions on~$X$, if $\seqn{f_n}{n}$ is $\I$-uniformly convergent to~$f$ then $\seqn{f_n}{n}$ is $(\I,\J)$-equally convergent to~$f$.
    \item For any sequence $\seqn{f_n}{n}$ of real-valued functions on~$X$, if $\seqn{f_n}{n}$ is $\sigma$-$\I$-uniformly convergent to~$f$ then $\seqn{f_n}{n}$ is $(\I,\J)$-equally convergent to~$f$.
\end{enumerate}
\end{theorem*}
This is an interesting result since in the classical (non-ideal) case, uniform convergence is stronger than equal convergence and $\sigma$-uniform convergence is even equivalent to equal convergence. 
\par
We shall study the~introduced variation of the P-ideal notion for ideals which are defined via a~given partition of natural numbers. Four such critical ideals, which obtained standard notation $\fino$, $\ED$, $\ofin$, $\finfin$, have been isolated in the literature (see below for the~definition and~\cite{BrFl17,BrFaVe,Hr,hrusakborel} for the details). We consider an additional ideal -- $\sel$, the~ideal generated by the family of graphs of functions from ${^\omega\omega}$. We have found all pairs $\I,\J$ such that $\I$ is a~$\mathrm{P}(\J)$-ideal, see the~table below or Table~\ref{omxom}.
 \begin{center}
\begin{tabular}{|c|c|c|c|c|c|c|}
\hline
 & $\mathrm{P}$  & $\mathrm{P}(\fin\times\emptyset)$&  $\mathrm{P}(\sel)$   & $\mathrm{P}(\E\D)$ &  $\mathrm{P}(\emptyset\times\fin)$ &  $\mathrm{P}(\fin\times\fin)$\\
 \hline\hline
  $\fin\times\emptyset$ & \xmark  & \cmark & \xmark &\cmark & \xmark & \cmark  \\
  \hline
  $\sel$ & \xmark  & \xmark & \cmark &\cmark & \cmark & \cmark\\
  \hline
  $\E\D$ & \xmark   & \xmark & \xmark &\cmark & \xmark &\cmark\\
  \hline
  $\emptyset\times\fin$& \cmark  &  \cmark & \cmark &\cmark& \cmark  & \cmark\\
  \hline
  $\fin\times\fin$ & \xmark  & \cmark & \xmark & \cmark & \xmark  &\cmark \\
  \hline
\end{tabular}
\end{center}
Furthermore, we study a~more general situation when the~partition~$\A$ inducing an ideal $\I$ is potentially different from the~partition~$\B$ inducing an ideal $\J$. Therefore we introduce the~following notation. We begin with non-tall ideals. Let $\A$ be an infinite partition of $\omega\times\omega$ into infinite sets.\footnote{ Throughout the paper, we sometimes use ``for any $\A$'' instead of more cumbersome ``for any infinite partition $\A$ of $\omega$ into infinite sets''. The intended meaning is usually clear from the context.}
\begin{center}
        \begin{tabular}{ll}
             $\finop{\A}$& the~ideal generated by~$\A$,\\[0.2cm]
             $\selp{\A}$& the~ideal generated by the~family of all selectors of~$\A$,\footnotemark\\[0.2cm]
             $\ofinp{\A}$& the~family of all sets with finite intersection with each element of~$\A$.
        \end{tabular}
        \footnotetext{We say that a set $S\in\PP(X)$ is a selector (semiselector) of an infinite family $\A\subseteq\PP(X)$, if $\abs{S\cap A}=1$ ($\abs{S\cap A}\leq 1$ and $S$ is infinite) for any $A\in\A$ and $S\setminus \bigcup\A=\emptyset$.}
\end{center}
We sometimes write shortly $\langle\A\rangle$ instead of~$\finop{\A}$. A~supremum of pairs of previous ideals leads to the~following tall ideals.
\begin{center}
        \begin{tabular}{ll}
             $\EDp{\A}$& the~ideal generated by~$\finop{\A}\cup\selp{\A}$,\\[0.2cm]             
             $\finfinp{\A}$& the~ideal generated by~$\finop{\A}\cup\ofinp{\A}$.
        \end{tabular}
\end{center}
One can see that the~introduced ideals are representations of all isomorphic copies of ideals $\fino$, $\sel$, $\ED$, $\ofin$, $\finfin$, respectively. Considering vertical sets $V_n=\{n\}\times\omega$,\footnote{We will use this notation for the sets $\{n \}\times \omega$ throughout the paper.} and corresponding partition $\V=\set{V_n}{n\in\omega}$, we obtain $\fino=\finop{\V}$, $\sel=\selp{\V}$, etc. 
\par
The~paper is organized as follows. In the~next section, we continue recalling the results on the P-property for two arbitrary ideals obtained by other authors. Our own observations are added as well. The~most important result there is the~introduction of the~relation~$\rsub{}{}$ between two ideals. It is a~natural sufficient condition for the~P-property for two ideals to be satisfied, and can be seen as a generalization of inclusion and orthogonality, see Proposition~\ref{ortaink}. Surprisingly, the~relation~$\rsub{}{}$ is a~characterization of a~P-property for many pairs of ideals we consider. Indeed, we say that $\rsub{\I}{\J}$ if the~whole ideal $\I$ has a~$\J$-union\footnote{A~set~$E$ such that $I\subseteq^\J E$ for all $I\in\I$.} in~$\I$, and by Corollary~\ref{prehlcor}, Theorems~\ref{ref1}, \ref{veze}, \ref{ed_ofin}, we obtain
\begin{theorem*}
If $\I$ is one of $\finop{\A}$, $\selp{\A}$, $\EDp{\A}$, and $\J$ one of $\finop{\B}$, $\ofinp{\B}$, then 
\begin{center}
    $\I$ is a $\mathrm{P}(\J)$-ideal if and only if $\rsub{\I}{\J}$.    
\end{center}
\end{theorem*}
\noindent Previous theorem is not true for $\I=\ofinp{\A}$ since $\ofinp{\A}$ is a~P-ideal. 
\par
The~reasoning for Table~\ref{omxom} is included in Section~\ref{natural_finfin}. Copies of~$\fino$ and~$\finfin$ are being examined in Section~\ref{natural_finfin} as well, and in Corollary~\ref{prehlcor} we prove the~following.
\begin{theorem*}
$\finop{\A}$ is a $\mathrm{P}(\J)$-ideal if and only if $\finfinp{\A}$ is a $\mathrm{P}(\J)$-ideal.
\end{theorem*}
\par
Considering copies of the~ideal $\sel$ in Section~\ref{S-towers} we found a~schema of combinatorial properties of partitions describing necessary criteria for $\selp{\A}$ to be a~$\pid{\J}$-ideal for all mentioned copies~$\J$ of ideals induced by a~partition, see Proposition~\ref{sufficient}. Moreover, the~conditions are sufficient in some cases as well. We say that a~family of $k$ disjoint semiselectors of~$\A$ is an~($\omega,k$)-tower of $\B$-monochromatic semiselectors of~$\A$ if all semiselectors from this family are selectors of a~single infinite subfamily of $\A$, and each semiselector is covered by a~set in~$\B$. In Theorem~\ref{veze} we show the~following.
\begin{theorem*}
$\selp{\A}$ is a $\pid{\ofinp{\B}}$-ideal if and only if there is no \emph{(}$\omega,k$\emph{)}-tower of $\B$-monochromatic semi\-se\-lectors of~$\A$ for some~$k$.\footnote{However, to simplify the~exposition we treat just $\A=\V$ in the~paper, which describes just when $\sel$ is a $\pid{\ofinp{\B}}$-ideal. In such a~case, semiselectors are just partial functions on~$\omega$. One can easily modify the~definition of the~notion to obtain a~more general situation.
}
\end{theorem*}
\par
In Section~\ref{towers_equiv}, we add copies of the~ideal $\ED$ into consideration. We show that in two cases, the~orthogonality is equivalent to the P-property for two ideals. Thus, in Theorems~\ref{ref1} and~\ref{ed_ofin} we prove the~following.
\begin{theorem*}
\
\begin{enumerate}[\rm (a)]
    \item $\EDp{\A}$ is a $\mathrm{P}(\finop{\B})$-ideal if and only if $\EDp{\A}\perp\finop{\B}$.
    \item $\EDp{\A}$ is a $\pid{\ofinp{\B}}$-ideal if and only if $\EDp{\A}\perp\ofinp{\B}$.
\end{enumerate}
\end{theorem*}
\par
Finally, the~last section is devoted to the~construction of a~partition~$\E$, and to the~non-trivial proof of $\sel$ being not a~$\pid{\selp{\E}}$-ideal, see Theorem~\ref{tazkaThm}. However, the~partition~$\E$ is constructed in a~way that restricts the presence of monochromatic towers to extremely small cases. Thus, our convenient combinatorial property is either not enough to characterize all cases in which the~ideal $\sel$ satisfies P-property for two ideals, or additional assumptions must be considered.


\section{Introducing P-property for two ideals}

The~aim of this~section is to recall basic terminology and outline the basic behavior of the P-property for two arbitrary ideals. All the~other sections are devoted to the study of the P-property for two ideals induced by partitions. Here, we shall summarize the~results appearing in the~literature, and we shall extend the list with our own observations.
\par
Let us begin with recalling basic terminology and notation on ideals.  We use $(\forall^\infty n\in\omega)$ to abbreviate ``for all but finitely many $n$'s'' and $(\exists^\infty n\in\omega)$ is an abbreviation for ``there exist infinitely many $n$'s''. Let $M$ be a non-empty infinite set. An \textbf{ideal} on $M$ is a family $\I\subseteq \PP(M)$ closed under taking subsets and finite unions, i.e., $\I$ is an ideal if $B\in\I$ for any $B\subseteq A\in\I$ and $A\cup B\in\I$ for any $A,B\in\I$. Moreover, we assume that all ideals are proper ($M\notin \I$) and contain all finite subsets of $M$. We are interested mostly in ideals on $\omega$. However, we study these ideals also on other countable sets $M$, e.g., $\omega\times\omega$, by identifying $M$ with $\omega$ via a fixed bijection. The ideal of all finite subsets of $M$ is denoted by $\fin$, where the underlying set $M$ is usually evident from the context. Calligraphic $\I,\J$ are used exclusively to denote ideals on a countable set. A~family $\D\subseteq\I$ is a \textbf{base} of $\I$ if for any $A\in\I$ there is $D\in\D$ such that $A\subseteq D$. $\D$ is a~\textbf{subbase} for $\I$ if the family of all finite unions of members of $\D$ is a base of $\I$. 
\par
For any $A, B\subseteq M$ and an ideal $\J$ on $M$, we write $A\subseteq^\J B$ if $A\setminus B\in\J$, and $A=^\J B$ if $A\subseteq^\J B$ and $B\subseteq^\J A$, i.e., if $A\triangle B\in\J$, where $\triangle$ stands for the symmetric difference. We use the notation $A\subseteq^*B$ instead of $A\subseteq^\fin B$. We recall from Definition~\ref{definiciaPJ} that $\I$ is a~$\boldsymbol{\mathrm{P}(\J)}$\textbf{-ideal}\footnote{We use the notation from \cite{filipow}, although the notion was introduced in \cite{macslez} as $\mathrm{AP}(\I,\J)$ property.} if for any family $\{I_n:n\in\omega \}\subseteq\I$ there is $I\in\I$ such that $I_n\subseteq^\J I$ for each $n\in\omega$. Note that for a~base~$\D$ of $\I$ we have that $\I$ is $\mathrm{P}(\J)$ if and only if for every family $\{D_n:n\in\omega \}\subseteq\D$ there is $D\in\D$ such that $D_n\subseteq^\J D$ for every~$n$. It is common to use the term $\mathrm{P}$-ideal instead of $\mathrm{P}(\fin)$-ideal. 
\par
We include a~list of reformulations of the~definition by M.~Mačaj and M.~Sleziak~\cite{macslez}. Let us recall that $\I^*$ is the~\textbf{dual filter} to $\I$, i.e., $\I^*=\{M\setminus I:I\in \I \}$.
\par
\begin{lema}[M. Mačaj -- M. Sleziak]
Let $\I$ and $\J$ be ideals on the same set $M$. The following conditions are equivalent:
\begin{enumerate}[(1)]
    \item $\I$ is a $\pid{\J}$-ideal.
   \item Any family $\{F_n:n\in\omega \}\subseteq\I^*$ has $\J$-intersection\footnote{We say that a set $X$ is $\K$-intersection of a family $\{X_n:n\in\omega \}$ if $X\subseteq^\K X_n$ holds for each $n\in\omega$. $\K$-intersection is "$\K$-pseudointersection" in \cite{macslez}. } in $\I^*$.
    \item For every family $\{I_n:n\in\omega \}\subseteq\I$ there is a family $\{J_n:n\in\omega\} \subseteq\I$ such that $I_n=^\J J_n$ for $n\in\omega$ and $\bigcup_{n\in\omega}J_n\in\I$. 
     \item For every family of mutually disjoint sets $\{I_n:n\in\omega \}\subseteq\I$ there is a family $\{J_n:n\in\omega\} \subseteq\I$ such that $I_n=^\J J_n$ for $n\in\omega$ and $\bigcup_{n\in\omega}J_n\in\I$. 
     \item For every family $\{I_n:n\in\omega \}\subseteq\I$ such that $I_n\subseteq I_{n+1}$ for each $n\in\omega$ there is a family $\{J_n:n\in\omega\} \subseteq\I$ such that $I_n=^\J J_n$ for $n\in\omega$ and $\bigcup_{n\in\omega}J
     _n\in\I$. 
    \item In the Boolean algebra $\mathcal{P}(M)/\J$ the ideal $\I$ corresponds to a $\sigma$-directed subset\footnote{I.e., it contains an upper bound of each countable subset.}.
\end{enumerate}
\end{lema}

\par
Some of the basic examples of pairs  of ideals such that $\I$ is a~$\mathrm{P}(\J)$-ideal were pointed out by R.~Filipów and M.~Staniszewski in \cite{stanfil}. Ideals $\I$, $\J$ on $M$ are said to be \textbf{orthogonal}, written $\I\perp\J$, if there is a set $A\subseteq M$ such that $A\in\I$ and $M\setminus A\in \J$.  An ideal $\I$ is \textbf{maximal}, if either $A\in \I$ or $M\setminus A\in \I$ for every $A\subseteq M$ (or equivalently, $\I$ is maximal with respect to ordering by inclusion).
\begin{lema}[R.~Filipów -- M.~Staniszewski]\label{ort}
\
\begin{enumerate}[(a)]
    \item If $\I$ is a $\mathrm{P}$-ideal, then $\I$ is a $\mathrm{P}(\J)$-ideal for every $\J$. 
    \item If $\I,\J$ are orthogonal then $\I$ is $\mathrm{P}(\J)$ and $\J$ is $\mathrm{P}(\I)$.
    \item If $\I,\J$ are maximal then $\I$ is $\mathrm{P}(\J)$ and $\J$ is $\mathrm{P}(\I)$.    
    \item If $\I\subseteq \J$, then $\I$ is a~$\mathrm{P}(\J)$-ideal. In particular, $\I$ is a~$\mathrm{P}(\I)$-ideal for every~$\I$.
    \item If $\J_1\subseteq\J_2$ and $\I$ is a~$\mathrm{P}(\J_1)$-ideal, then $\I$ is a~$\mathrm{P}(\J_2)$-ideal.
\end{enumerate}    
\end{lema}
\noindent Taking into account Lemma~\ref{ort}(e), and assuming $\I$ is a~$\mathrm{P}(\J)$-ideal, it would be tempting to conjecture that either $\I'\subseteq\I$ or $\I'\supseteq\I$ leads to $\I'$ being a~$\mathrm{P}(\J)$-ideal. However, it is not necessarily true, see Table~\ref{omxom}. 
\par
In the~following proposition, we generalize parts~(b) and~(d) of Lemma~\ref{ort}, introducing a~unifying element, the~relation~$\rsub{}{}$. Given an~ideal $\I$ on $M$ we denote by $\I^+$ the family of $\boldsymbol{\I}$\textbf{-positive sets} (subsets of $M$ that are not in $\I$), i.e., $\I^+=\PP(M)\setminus\I$. If $X\in\I^+$, we denote by $\I\restriction X$ the ideal $\{I\cap X: I\in\I \}$ on~$X$. Recall that a set $E$ is a $\boldsymbol{\J}$\textbf{-union} of $\I$ if $I\subseteq^\J E$ for each $I\in \I$. Then  we write $\boldsymbol{\rsub{\I}{\J}}$ if the~whole ideal $\I$ has a~$\J$-union in~$\I$. Finally, notice that $\rsub{\I}{\J}$ if and only if
\[
(\exists E\in\I^*)\ \I\restriction E\subseteq\J,
\] 
so if $\I$ is almost a subset of $\J$ in a  sense. We use these two equivalent definitions interchangeably throughout the text. However, we prefer to work with the latter one most of the time.

Notice that $\I\subseteq\J$ immediately implies $\rsub{\I}{\J}$. 
\begin{prop}\label{ortaink}
\
\begin{enumerate}[(a)]
    \item If $\I\perp\J$ then $\rsub{\I}{\J}$.
    \item If $\rsub{\I}{\J}$ then $\I$ is a $\mathrm{P}(\J)$-ideal.
\end{enumerate}    
\end{prop}
\begin{proof}
(a) If $\I\perp \J$ then there is a set $E'\in\I$ such that $E=M\setminus E'\in\J\cap\I^*$.

(b) Straightforward.

\end{proof}
Note that $\rsub{\I}{\J}$ is not necessarily equivalent to $\I\subseteq\J\ \lor\  \I\perp\J$. Before we give a~counterexample, let us recall that
if there is not a finite subfamily $\mathcal{C}$ of a~family $\D$ with finite $M\setminus \bigcup \C$ then \textbf{ideal generated by} $\boldsymbol{\D}$ is the smallest ideal containing $\E$ and $\fin$, and we denote this ideal by $\gen{\D}$. So
$$\gen{\mathcal{D}}=\left\{D\in\mathcal{P}(M): D\subseteq^* \bigcup \mathcal{D}'\text{ for some }  \mathcal{D}'\in[\mathcal{D}]^{<\omega}  \right\}.$$ 
 Note that $M=\bigcup \D$ if and only if $\D$ is a subbase for $\gen{\D}$. The family $\D\cup\fin$ is always a subbase for $\gen{\D}$.
Returning back to the~counterexample, let $\I=\gen{\{K\}}$ for some infinite and coinfinite $K\subseteq \omega$. Since $\omega\setminus K\in \I^*$ and $\I\restriction (\omega\setminus K)\subseteq\fin$, we have $\rsub{\I}{\fin}$. Evidently, $\I\not\subseteq \fin$ and $\fin\not\perp\I$ since $\I$ does not contain any cofinite set. Clearly, the condition $\rsub{\I}{\J}$ is equivalent to 
\[
(\forall\kappa\leq\abs{\I})(\forall \mathcal{E}\in[\I]^\kappa)(\exists E\in \I^*)\ \mathcal{E}\restriction E\subseteq \J.
\]
Actually, using the definition and notation from~\cite{stan}, the~condition $\rsub{\I}{\J}$ is equivalent to any of the~following conditions:\footnote{A notion of a $\kappa\text{-}\mathrm{P}(\J,\I)$-ideal in \cite{stan} is a~modification of the $\pid{\J}$-ideal notion. Let $\I,\J,\K$ be ideals on $M$ and $\kappa$ be a cardinal number. We say that an ideal $\K$ is a $\kappa$-$\mathrm{P}(\J,\I)$-ideal, if for every family $\{K_\alpha\subseteq M:\alpha<\kappa \}\subseteq \K$ there exists $I\in \I$ such that $K_\alpha\setminus I\in \J$ for every $\alpha<\kappa$.} $\I$ is a $\abs{\I}\text{-}\mathrm{P}(\J,\I)$-ideal, $\I$ is a $\mathtt{cof}(\I)\text{-}\mathrm{P}(\J,\I)$-ideal,\footnote{Let us recall that $\mathtt{cof}(\I)=\min\{\abs{\D}:\D\subseteq\I\wedge\D\text{ is a base of }\I \}$.} $\I$ is a~$\kappa\text{-}\mathrm{P}(\J,\I)$-ideal for every~$\kappa$.
\par
We say that an ideal $\I$ on $M$ is \textbf{tall} if for each $X\in[M]^\omega$ there exists $I\in\I$ such that $\abs{X\cap I}=\omega$.  An~ideal $\I$ is \textbf{nowhere tall}, if 
\[
(\forall A\in\I^+)(\exists B\in[A]^\omega)\ \I\restriction B=[B]^{<\omega},
\]
for more see \cite{filtry}. Clearly, $\fino$ and $\ofin$ are nowhere tall ideals. If an ideal $\I$ is tall and an ideal $\J$ is nowhere tall, then~$\rsub{}{}$ is equivalent to the orthogonality:

\begin{lema}\label{nowheretall}
Let $\I,\J$ be ideals on $M$ such that $\I$ is tall and $\J$ is nowhere tall. The following statements are equivalent.
\begin{enumerate}
    \item[\rm (1)] $\I\perp\J$.
    \item[\rm (2)] $\rsub{\I}{\J}$.
\end{enumerate}
\end{lema}
\proof Part (1)$\to$(2) follows from Proposition~\ref{ortaink}(a).

$\neg$(1)$\to\neg$(2). If $E\in \I^*$ then $E\notin \J$ by the assumptions. Ideal $\J$ is nowhere tall and therefore there is infinite $E'\subseteq E$ such that $\J\restriction E'=[E']^{<\omega}$. Since $\I$ is tall, there is an infinite $ D\subseteq E'$ such that $D\in \I$. Then $D\in \I\restriction E$ and $D\notin \J$, i.e., $\I\restriction E\not\subseteq\J$.
\qed

It may seem counterintuitive that a tall ideal can be ``almost'' a subset of an obviously less ``dense'' nowhere tall ideal but this situation may occur. Consider, e.g., bijective function $f\colon\omega\times\omega\to\omega\times\omega$ such that $f(\ED)\perp\fino$. In this case $\rsub{f(\ED)}{\fino}$.

Let $\I,\J$ be ideals on $M$. One can easily see that if ideals $\I,\J$ are not orthogonal then the set 
$$
\I\lor \J=\{I\cup J: I\in\I, J\in\J \}
$$
is an ideal on $M$. In fact, $\lor$ is supremum of two non-orthogonal ideals in the family of all ideals on $M$ ordered by inclusion. 

We have already commented that if $\I_1$ is a $\pid{\J}$-ideal and $\K\supseteq \I_1$ then it is not necessarily true that $\K$ is a $\pid{\J}$-ideal, see Table~\ref{omxom}. However, we show that if $\K=\I_1\lor\I_2$ and $\I_2$ is also a $\pid{\J}$-ideal
then it is the case. Note that the~assumption about $\I_1,\I_2$ not being orthogonal should be read as $\I_1\lor\I_2$ is an ideal. Finally, note that if $\K\supseteq \I_1$ then $\K=\I_1\lor\K$, $\I_1,\K$ are non-orthogonal, but it is not necessarily true that $\K\cap \I_1\subseteq\J$.

\begin{prop}\label{differJ}
Let $\J$ be an ideal on $M$ and let $\I_1,\I_2$ be non-orthogonal ideals on $M$ such that $\I_1\cap \I_2\subseteq\J$. The following statements are equivalent.
\begin{enumerate}
    \item[\rm (1)] $\I_1$ and $\I_2$ are $\mathrm{P}(\J)$-ideals. 
    \item[\rm (2)] $\I_1\lor \I_2$ is a~$\mathrm{P}(\J)$-ideal.
\end{enumerate}
\end{prop}
\proof (1)$\to$(2) For any family $\{K_n:n\in\omega \}\subseteq \I_1\lor \I_2$ there are families $\{I^1_n\in\I_1:n\in\omega \}\subseteq\I_1$ and $\{I^2_n\in\I_2:n\in\omega \}\subseteq\I_2$, such that $K_n=I^1_n\cup I^2_n$ for each $n$. There are $I^1\in\I_1$, $I^2\in\I_2$ such that $I^1_n\setminus I^1\in\J$ and $I^2_n\setminus I^2\in\J$ for each $n$. Then $K_n\setminus K\in \J$ for each $n$, where $K=I^1\cup I^2$.

(2)$\to$(1) Let $\{I_n:n\in\omega \}\subseteq \I_1\subseteq \I_1\lor\I_2$. There is $I^1\in\I_1$ and  $I^2\in\I_2$, such that $I_n\setminus \left(I^1\cup I^2\right)\in \J$ for each $n$. For every $n\in\omega$
$$I_n \setminus I^1\subseteq \left( I_n\setminus \left(I^1\cup I^2\right) \right)\cup \left(I_n\cap I^2\right).$$
$I_n\setminus  \left(I^1\cup I^2\right)\in\J$ because we assume (2) and $\left(I_n\cap I^2\right)\in\J$ from the assumption $\I_1\cap \I_2\subseteq \J$. Thus, $I_n\setminus I^1\in\J$ for every $n\in\omega$. Proof for the ideal $\I_2$ is completely analogous.
\qed

Note that the implication (1)$\to$(2) of Proposition~\ref{differJ} does not depend on the assumption $\I_1\cap\I_2\subseteq\J$. On the other hand, it is crucial for the implication (2)$\to$(1). The counterexample are ideals $\sel$ and $\ofin$ studied in the next section.

By Proposition~\ref{differJ} we obtain the~following.
\begin{corol}\label{finfin}
Let $\I_1,\I_2$ be non-orthogonal ideals on a set $M$ such that $\I_2$ is a $\mathrm{P}(\I_1)$-ideal. Then $\J=\I_1\lor\I_2$ is a $\mathrm{P}(\I_1)$-ideal.
\end{corol}
\par
We conclude the~section with the~connections to cardinal invariants. It can be easily seen that $\I$ is a $\pid{\J}$-ideal if and only if $\mathfrak{b}(\I,\subseteq^\J)\geq\omega_1$\footnote{For any preordered set without maximal elements $\langle X,\leq \rangle$, let
$$
\mathfrak{b}(X,\leq)=\min\{\abs{A}:A\subseteq X, A\text{ is unbounded from above in  }X \}.
$$
} 
(or $\I$ is not a $\pid{\J}$-ideal if and only if $\mathfrak{b}(\I,\subseteq^\J)=\omega$). Note that for cardinal invariant $$\mathtt{add}^*(\I)=\min\{\abs{\E}:\E\subseteq\I \land (\forall X\in \I)(\exists E\in \E)\ E\not\subseteq^* X \}$$ studied in \cite{inv},  we have $\mathfrak{b}(\I,\subseteq^\fin)=\texttt{add}^*(\I)$,  thus the invariant $\mathfrak{b}(\I,\subseteq^\J)$ can be seen as $``\mathtt{add}^\J(\I)"$ -- generalization of the invariant $\mathtt{add}^*(\I)$.


\section{Natural partition and copies of product ideals}\label{natural_finfin}

For each pair among ideals $\fin\times\emptyset$,  $\sel$, $\ED$, $\emptyset\times\fin$, $\finfin$, we shall decide whether it satisfies the P-property for two ideals or not, see Table~\ref{omxom}. The~second part of the~section is an~initiation of the~study of the~same question for isomorphic copies of the above-mentioned ideals. Although all five ideals have been already discussed in the~introduction, including their definition, we include a~more detailed treatment before we reach the~main result of the~present section. Therefore, let us recall yet another viewpoint on their definition not included in the~introduction.
\par
By $\emptyset$ we denote the~ideal $\{ \emptyset\}$. Here we violate our assumption that $\fin$ is a subset of every ideal, but this notation comes in useful, especially when dealing with Fubini products of ideals. Given two ideals $\I,\J$ on~$\omega$, the \textbf{Fubini product} $\I\times\J$ is defined by
$$\I\times\J=\{A\subseteq\omega\times\omega: \{n:\{m:\langle n,m\rangle\in A \}\notin\J \} \in\I\}.$$
Thus, $\fin\times\emptyset$, $\emptyset\times\fin$, $\finfin$ can be seen also as Fubini products of ideals in $\{\fin, \emptyset\}$. This is, in fact, a justification for the~choice of the notation. A~direct definition of $\sel$ and $\ED$ can be expressed as follows:
\begin{align*}
 \sel=&\ \{X\subseteq \omega\times\omega: (\exists m\in\omega)(\forall k\in\omega)\ \abs{\{l:\ev{k,l}\in X\}}\leq m  \},\\
\ED=&\ \{X\subseteq \omega\times\omega: (\exists m\in\omega)(\forall^\infty k\in\omega)\ \abs{\{l:\ev{k,l}\in X\}}\leq m  \}\\
 = &\ \{X\subseteq\omega\times\omega: \limsup_{k\to\infty}\abs{(\{ k\}\times\omega)\cap X}<\infty \}.
\end{align*}

Ideal $\E\D$ can be thought of as an ideal on $\omega$ generated by infinite partition of $\omega$ into infinite sets and selectors of this partition. Considering the ideal $\sel$, if we identify functions with their graphs, we get another equivalent definition $\sel=\gen{{^\omega\omega}}$. 

The~following diagram shows relationships between ideals with respect to the inclusion relation.
\begin{center}
\begin{figure}[H]
\centering
\begin{tikzpicture}
  \matrix (m) [matrix of math nodes, row sep=3em,
    column sep=3em]{
    \fino&\ED&\finfin&\\
  \fin & \sel &\ofin&\\};
  \path[-stealth]
  (m-1-1) edge (m-1-2)
  (m-2-1) edge (m-1-1) edge (m-2-2)
  (m-2-2) edge (m-1-2) edge (m-2-3)
  (m-1-2) edge (m-1-3)
  (m-2-3) edge (m-1-3);
\end{tikzpicture}
   \caption{Inclusions between critical ideals on $\omega\times\omega$.}
   \label{inclusion} 
 \end{figure}
\end{center}

Moreover, recall that 
\begin{center}
    $\sel=\ED\cap(\ofin)$,\hspace{0.5cm}    $\ED=(\fin\times\emptyset)\lor\sel$,\hspace{0.5cm} $\fin\times\fin=(\fin\times\emptyset)\lor (\emptyset\times\fin)$.
\end{center}

The~following table shows which pairs~$\I,\J$ among $\fin\times\emptyset$, $\emptyset\times\fin$, $\finfin$,  $\sel$, and $\ED$ satisfy that $\I$ is a~$\pid{\J}$-ideal. The~reasoning can be found below Table~\ref{omxom}. In fact, many of the results come from the~fact that $\I\subseteq \J$ yields $\I$ is a $\pid{\J}$-ideal.

 \begin{center}
 \begin{table}[H]
\centering
\begin{tabular}{|c|c|c|c|c|c|c|}
\hline
 & $\mathrm{P}$  & $\mathrm{P}(\fin\times\emptyset)$&  $\mathrm{P}(\sel)$   & $\mathrm{P}(\E\D)$ &  $\mathrm{P}(\emptyset\times\fin)$ &  $\mathrm{P}(\fin\times\fin)$\\
 \hline\hline
  $\fin\times\emptyset$ & \xmark  & \cmark & \xmark &\cmark & \xmark & \cmark  \\
  \hline
  $\sel$ & \xmark  & \xmark & \cmark &\cmark & \cmark & \cmark\\
  \hline
  $\E\D$ & \xmark   & \xmark & \xmark &\cmark & \xmark &\cmark\\
  \hline
  $\emptyset\times\fin$& \cmark  &  \cmark & \cmark &\cmark& \cmark  & \cmark\\
  \hline
  $\fin\times\fin$ & \xmark  & \cmark & \xmark & \cmark & \xmark  &\cmark \\
  \hline
\end{tabular}
   \caption{$\mathrm{P}(\J)$ interactions between critical ideals on $\omega\times\omega$ }
   \label{omxom} 
 \end{table}
\end{center}

Below we comment on the results in Table~\ref{omxom} row by row. In fact, all the~positive statements are based on one of the three following reasons or their combinations: inclusion between ideals (Proposition~\ref{ort}(d), and Diagram~\ref{inclusion}), the~fact that $\ofin$ is a $\mathrm{P}$-ideal (see, e.g., \cite{Hr}), and finally, $\fin\times\fin$ is an ideal $(\fin\times\emptyset)\lor(\emptyset\times\fin)$ (Lemma~\ref{finfin}). In particular, the~first among the~latter reasons yields that $\I$ is a~$\pid{\I}$-ideal, and therefore these pairs are excluded from the~further considerations (a~diagonal in Table~\ref{omxom}). In case of negative statement, the~corresponding reasoning is included.
\begin{itemize}
    \item Since $\ofin$ is a $\mathrm{P}$-ideal, it is also $\pid{\I}$ for all ideals $\I$ on $\omega\times\omega$ by Lemma~\ref{ort}(d).
    \item $\fin\times\emptyset$ is a subset of $\fin\times\fin$ and also $\ED$, therefore it is a $\pid{\finfin}$-ideal and $\pid{\ED}$-ideal. A sequence $A_n=\{n \}\times\omega$ is a witness for $\fin\times\emptyset$ being not a $\mathrm{P}(\emptyset\times\fin)$-ideal and since $\sel\subseteq\emptyset\times\fin$, we have that $\fin\times\emptyset$ is not a $\pid{\sel}$-ideal either.
    \item $\fin\times\fin$ is an ideal $(\fin\times\emptyset)\lor(\emptyset\times\fin)$. It follows from Lemma~\ref{finfin} that $\fin\times\fin$ is a $\mathrm{P}(\fin\times\emptyset)$-ideal and since $\fin\times\emptyset\subseteq\E\D$ we have that $\fin\times\fin$ is also a $\mathrm{P}(\E\D)$-ideal. A sequence $A_n=\{n \}\times\omega$ is a witness for $\fin\times\fin$ being not a $\mathrm{P}(\emptyset\times\fin)$-ideal and hence $\fin\times\fin$ is not a $\pid{\sel}$-ideal.
   \item $\sel$ is a subset of $\ofin$, $\finfin$ and $\ED$ and hence is a $\pid{\ofin}$, $\pid{\finfin}$ and $\pid{\ED}$-ideal. A sequence $\omega\times\{n\}$ is a witness for $\sel$ not being a $\pid{\fin}$ and $\pid{\fino}$-ideal. 
    \item $\ED$ is a subset of the ideal $\fin\times\fin$ so $\ED$ is a $\mathrm{P}(\fin\times\fin)$-ideal. A sequence $A_n=\{ n \}\times\omega$ is a witness for $\E\D$ not being a $\mathrm{P}(\emptyset\times\fin)$-ideal and hence $\E\D$ is not a $\pid{\sel}$-ideal. $\E\D$ is not a $\mathrm{P}(\fin\times\emptyset)$-ideal either and a witness is a sequence $A_n=\omega\times\{n \}$.
\end{itemize}

The~rest of the~paper is devoted to an~investigation of when the~isomorphic copies of ideals $\fin\times\emptyset$, $\emptyset\times\fin$, $\finfin$,  $\sel$, and $\ED$ satisfy P-property for two ideals. Let us recall that ideals $\I$ and $\J$  on $X$ and $Y$, respectively, are \textbf{isomorphic}, written $\I\simeq\J$, if there exists a bijection $f\colon X\to Y$, such that $I\in\I\equiv f[I]\in\J$ for every $I\in\I$ (or equivalently $\J=f(\I)=\{f[I]:I\in\I \}$) and we call such a bijection an~\textbf{isomorphism}.
\par
The~situation rapidly changes when it comes to isomorphic copies in comparison with original ideals $\fin\times\emptyset$, $\emptyset\times\fin$, $\finfin$,  $\sel$, and $\ED$. Now, before the examination of the situation, we shall make an important remark -- we state our results for ideals defined via arbitrary partition $\A$ (or $\B$), and a particular vertical partition $\V$. However, this is not a setback since the results are true for two arbitrary partitions, i.e., for the general formulation one can think of $\V$ being an arbitrary partition.

Let us begin with a~basic observation. 
\begin{lema}\label{existujeInkIzo}
Let $M,N$ be countable sets. For each pair of ideals $\I$ on $M$ and $\J\neq \fin$ on $N$ such that $\I$ is not tall there is an isomoprhism $f$ such that $f(\I)\subseteq \J$.
\end{lema}
\proof Since $\I$ is not tall there is an infinite set $X\subseteq M$ such that $\I\restriction X=[X]^{<\omega}$. Consider arbitrary $f$ sending (bijectively) $M\setminus X$ to a set belonging to $\J$ and $X$ to the rest of $N$. \qed

Moreover, note that for each pair of ideals $\I,\J\neq\fin$ on $M$ there is an isomorphism $f$ such that $f(\I)$ is a $\mathrm{P}(\J)$-ideal. Indeed, let $I\in\I$ be an infinite set and $J\in\J$ be a coinfinite set. It suffices to consider a mapping $f$ such that $f[I]=M\setminus J$ to see that $f(\I)$ and $\J$ are orthogonal and thus $f(\I)$ is a $\mathrm{P}(\J)$-ideal. Consequently, for every $\J\neq\fin$ there are partitions $\A,\B,\mathcal{C},\mathcal{D}$ of $\omega\times\omega$ such that $\selp{\A}$, $\EDp{\A}, \finop{\B},\ofinp{\mathcal{C}}$ and $\finfinp{\mathcal{\mathcal{D}}}$ are $\mathrm{P}(\J)$-ideals. 
\par
Furthermore, the ideal $\ofin$ is a $\mathrm{P}$-ideal and therefore $\ofin$\footnote{as well as every isomorphic copy of $\ofin$, since isomorphisms preserve $\mathrm{P}$-property.} is a $\pid{\J}$-ideal for each ideal $\J$ on $\omega\times\omega$. Thus, $\ofinp{\A}$ is a~$\pid{\J}$-ideal for each ideal $\J$ on $\omega\times\omega$. For this reason, the ideal $\ofinp{\A}$ is omitted from further considerations.
 
For the next part we will need the following notions. A family $\A\subseteq [M]^\omega$ is an almost disjoint family (AD family) if $A\cap B\in\fin$ for each $A,B \in \A$, $A\neq B$. An AD family is a MAD family if for each infinite $X\subseteq M$ there is an $A\in\A$ such that $X\cap A$ is infinite (i.\,e.,\ $\A$ is $\subseteq$-maximal among the AD families). It is easy to see that every ideal $\I$ with a countable base $\D=\{D_n:n\in\omega \}$ has a countable subbase $\Ss=\{S_n:n\in\omega \}$ with $S_i\cap S_j=\emptyset$ whenever $i\neq j$. Indeed, let
$S_0=D_0$ and $S_k=D_k\setminus \bigcup_{l<k} D_l$ for $k>0$. Hence, every ideal with countable base is generated by an AD family. Note that an~ideal $\gen{\C}$ generated by a~family $\C$ is included in ideal~$\J$ if and only if $\C\subseteq\J$. 
 
 \begin{prop}\label{generovaneAD}
Let $\J$ be an ideal on $M$ and $\C$ be an infinite AD family. The following statements are equivalent.
\begin{enumerate}
    \item[\rm (1)] $\gen{\C}$ is a $\mathrm{P}(\J)$-ideal.
    \item[\rm (2)] $\C\setminus\J$ is finite.
      \item[\rm (3)] $\rsub{\gen{C}}{\J}$.
\end{enumerate}
\end{prop}
\proof

$\neg$(2)$\to\neg$(1) Let $\abs{\C}=\kappa$ and $\C=\{C_\alpha:\alpha<\kappa \}$ be an enumeration of $\C$. Let $\abs{\{\alpha\in\kappa:C_\alpha\notin\J \}}\geq\omega$. Choose an increasing sequence $\alpha_k$ such that $C_{\alpha_k}\notin\J$ for all $k$. For any $I\in\gen{C}$ pick a $k$ such that $\abs{I\cap C_{\alpha_k}}<\omega$. 

(2)$\to$(3) Let  $\gen{C}^*\ni E=\omega\setminus \bigcup (\C\setminus\J)$. Then $E$ is a witness for (3).

(3)$\to$(1) A consequence of Proposition~\ref{ortaink}(b).
\qed
 
If we combine the~Proposition~\ref{differJ} and Proposition~\ref{generovaneAD} we get an equivalence of the statements in Corollary~\ref{prehlcor}. Note that $\finop{\A}$ is included in an ideal~$\J$ as a subset if and only if $A_n\in \J$ for all $n\in\omega$, compare with the condition (3) in Corollary~\ref{prehlcor}. On the other hand $\finfinp{\A}$ is not a subset of $\J$ for $\J=\fino,\ED,\ofin$ for any~$\A$. Indeed, the statement for $\fino$ and $\ofin$ follows from the tallnes of $\finfinp{\A}$. In fact, if $\I$ is an ideal on $M$ and $\J\neq\fin$ is not a tall ideal, then $\I$ is tall if and only if $f(\I)\not\subseteq \J$ holds for all bijective mappings $f$. ``Only if" part is trivial. The contrapositive of ``if" part follows from Lemma~\ref{existujeInkIzo}.  The ideal $\finfin$ is not Katětov below $\ED$, see \cite{hrusakborel}, hence $\finfinp{\A}\not\subseteq\ED$. Finally, the~statement $\rsub{\finfinp{\A}}{\J}$ cannot be added to the~list of equivalent conditions in Corollary~\ref{prehlcor} since taking $A\in(\finfin)^\ast$, the~ideal $(\finfin)\restriction A$ is tall on $A$, but for $(\fino)\restriction A$ that is not the~case.

\begin{corol}\label{prehlcor}
Let $\J$ be an ideal on $\omega\times\omega$. The following statements are equivalent.
\begin{itemize}
    \item[\rm (1)] $\finop{\A}$ is a $\mathrm{P}(\J)$-ideal.
    \item[\rm (2)] $\finfinp{\A}$ is a $\mathrm{P}(\J)$-ideal.
    \item[\rm (3)] $(\forall^\infty n\in\omega)\ A_n\in \J$.
      \item[\rm (4)] $\rsub{\finop{\A}}{\J}$.
\end{itemize}
\end{corol}
\proof
(1)$\equiv$(2) Particular case of Proposition~\ref{differJ} for $\I_1=\finop{\A}$ and $\I_2=\ofinp{\A}$ using the fact that $\ofinp{\A}$ is a~P-ideal. 
\par
(1)$\equiv$(3) Particular case of Proposition~\ref{generovaneAD} for $\gen{\A}=\finop{\A}$.
\par
(3)$\to$(4) Straightforward.
\par
(4)$\to$(1) This implication follows from Proposition~\ref{ortaink}(b).
\qed

Let us focus on a particular case of Corollary~\ref{prehlcor} for $\J=\fino$. We offer two more equivalent conditions for $\finop{\A}$ to be a~$\mathrm{P}(\fino)$-ideal in Corollary~\ref{ekvivfub}. Note that $\finop{\A}$ is included in ideal~$\fino$ if and only if each~$A_n$ has nonempty intersection with only finitely many $V_m$'s. Compare with the
part~(2) of the~following.
\begin{corol}\label{ekvivfub}
The following statements are equivalent.
\begin{enumerate}[(1)]
    \item $\finop{\A}$ is a $\mathrm{P}(\fino)$-ideal.
    \item $(\forall^\infty n)(\forall^\infty m)\ A_n\cap V_m=\emptyset$. 
    \item $(\exists F\in\fin^*)\ \finop{\A}\restriction\bigcup_{k\in F}A_k\subseteq\fino$.
    \item $(\forall^\infty n)\ A_n\in\fino$.
\end{enumerate}
\end{corol}
\proof Straightforward.\qed

Table~\ref{omxom1} lists particular cases of Corollary~\ref{prehlcor} for ideal $\J$ being one of the studied ideals $\fin\times\emptyset$,  $\sel$, $\ED$, $\emptyset\times\fin$, $\finfin$. Note that Corollary~\ref{prehlcor} states that $\finop{\A}$ is a $\mathrm{P}(\J)$-ideal if and only if $\finfinp{\A}$ is a $\mathrm{P}(\J)$-ideal. Therefore, the~conditions in Table~\ref{omxom1} are valid for $\finop{\A}$ as well as for $\finfinp{\A}$. For typographical reasons we omit ``for all but finitely many $n\in\omega$" from conditions in Table~\ref{omxom1}. Note that the~statement with ``for all $n$" is equivalent to $\finop{\A}\subseteq\J$ for ideal $\J$ being one of the ideals in the~table. 

 \begin{center}
 \begin{table}[H]
\centering
\begin{tabular}{|c|c|c|c|c|c|}
\hline
 & $\mathrm{P(\finop{\B})}$ & $\mathrm{P}(\selp{\B})$ & $\mathrm{P}(\EDp{\B})$ & $\mathrm{P}(\ofinp{\B})$ & $\mathrm{P}(\finfinp{\B})$ \\
 \hline\hline
 $\finop{\A}$&  {\footnotesize $(\forall^\infty m)$}&  {\footnotesize $(\exists k)(\forall  m)$ } &{\footnotesize $(\exists k)(\forall^\infty m)$ } & {\footnotesize $(\forall m)$} &   {\footnotesize $(\forall^\infty m)$} \\
  $\finfinp{\A}$  & {\footnotesize $A_n\cap B_m=\emptyset$} & {\footnotesize $\abs{A_n\cap B_m}<k$} & {\footnotesize $\abs{A_n\cap B_m}<k$} & {\footnotesize $\abs{A_n\cap B_m}<\omega$} & {\footnotesize $\abs{A_n\cap B_m}<\omega$} \\
\hline
\end{tabular}
   \caption{The statements hold for all but finitely many $n\in\omega$.}
   \label{omxom1}
 \end{table}
\end{center}

\section{Selector ideal and towers of functions}\label{S-towers}

In this section we introduce a somewhat instinctive set-theoretic object, which, as we shall see, is tightly interrelated with the studied P-like property. The goal is to find a description of relationships between ideals using natural notions such as partial functions.
It turned out that we are able to formulate necessary conditions for $\selp{\A}$ to be a~$\pid{\J}$-ideal where $\J$ is one of the~ideals among $\fino$,  $\sel$, $\ED$, $\ofin$, $\finfin$ using so-called towers of functions. In the~following sections we shall show that towers of functions may be or may be not  a~sufficient tool for characterizing the studied P-notion. 

If $\B$ is an infinite partition of $\omega\times\omega$ into infinite sets, we call a partial function $f\subseteq\omega\times\omega$ $\boldsymbol{\B}$\textbf{-monochromatic} if there is $B\in\B$ such that $f\subseteq B$.
\begin{defi}
Let $\kappa,\lambda\leq \omega$ be cardinal numbers and $\B$ be an~infinite partition of $\omega\times\omega$ into infinite sets. A~set $\T$ of nonempty partial funcions from $\omega$ to $\omega$ is called $\boldsymbol{(\kappa,\lambda)}$\textbf{-tower of functions} ($\boldsymbol{(\kappa,\lambda)}$\textbf{-TF} for short) if
\begin{itemize}
    \item $(\exists a\in[\omega]^\kappa)(\forall g\in \T)\ \dom{g}=a$,
    \item $\abs{\T}=\lambda$,
    \item $g\neq g'\in\T\ \to\ g\cap g'=\emptyset$.
\end{itemize}
If, in addition, each $g\in\T$ is $\B$-monochromatic, we call such a $\tower{\kappa}{\lambda}$ a $\boldsymbol{(\kappa,\lambda)}$\textbf{-tower of} $\boldsymbol{\B}$\textbf{-mo\-no\-chroma\-tic functions} ($\boldsymbol{\tow{\kappa}{\lambda}}$ \textbf{w.r.t. }$\boldsymbol{\B}$ for short).
\end{defi}

 Briefly speaking, $\kappa$ is the~size of the~domain of functions in~$\T$, and $\lambda$ is the~size of~$\T$ itself. Sometimes we do not stress parameters~$\kappa,\lambda,\B$, and we say just towers of monochromatic functions or even just towers. Note that $\emptyset$ is the only $\tow{\kappa}{0}$ for every cardinal number~$\kappa$. On the other hand, there is no $\tow{0}{\lambda}$ for any~$\lambda>0$.

If $\T$ is a $\tow{\kappa}{\lambda}$, we use the following notation:
\begin{itemize}
   \item If $\T\neq\emptyset$, then $C_\T=\{B\in\B : (\exists g\in\T)\ g\subseteq B\}$ and
  $\dom{\T}=\dom{g}$, where $g\in\T$.
    \item If $\T=\emptyset$, then $C_\T=\emptyset$ and $\dom{\T}=\emptyset$.
\end{itemize}


We use the term  $\B$-monochromatic (or just monochromatic) because the partition $\B$ can be seen as a coloring of $\omega\times\omega$. Note that the existence of $\tow{\kappa}{\lambda}$ $\T$ implies the existence of $\tow{k}{l}$ $\T'$ with $\bigcup \T'\subseteq \bigcup \T$ for all $k<\kappa,l<\lambda$. We shall be interested in sequences of towers. Thus, we say that a~sequence $\seqn{\T_k}{k}$ of towers is a~sequence of \textbf{essentially different} towers if $C_{\T_i}\cap C_{\T_j}=\emptyset$ and $\dom{\T_i}\cap\dom{\T_j}=\emptyset$ for each $i\neq j$.
\par
Before we show the real connection between towers of functions and the P-property for two ideals, we will need the following simple lemmas.
\begin{lema}\label{omegaveze}
 For any finite $\F\subseteq{^\omega\omega}$ and any $\tower{\omega}{\abs{\F}+1}$ $\T$ there is $g\in \T$ such that $\abs{g\setminus\bigcup\F}=\omega$. 
\end{lema}
\begin{proof}
By contradiction. Let $\F\subseteq {^\omega\omega}$ be finite and let $\T$ be an~$\tower{\omega}{\abs{\F}+1}$ such that  $\abs{g\setminus \bigcup \F}<\omega$ for every $g\in \T$. Then $g\subseteq^* \bigcup \F$ for each $g\in \T$. So there is $k\in\omega$ such that $\bigcup_{g\in \T} g\restriction\{i>k:i\in\dom{\T} \}\subseteq \bigcup \F$. This is, of course, a contradiction, since a union of $\abs{\F}+1$ disjoint partial functions on $\{i>k:i\in\dom{\T} \}$ cannot be covered by $\abs{\F}$ functions.
\end{proof}
Consequently, for any $A\in \sel$ covered by $n$ functions and for any $\tower{\omega}{n+1}$ $\T$ there is a $g\in\T$ such that $\abs{g\setminus A}=\omega$.

\begin{lema}\label{kkveze}
Let $\T_k=\{g_i^k:i<k \}$ be a $\tower{k}{k}$ for every $k\in\omega$, such that $\dom{\T_i}\cap\dom{\T_j}=\emptyset$ for $i\neq j$ and let $\F\subseteq{^\omega\omega}$ be finite. Then there is $i\leq \abs{\F}$ such that
$$(\forall m\in\omega)(\exists n\geq m)\ \abs{g_i^n\setminus \bigcup\F}\geq m.$$
\end{lema}
\begin{proof}
    Let wlog $m\geq 1$ and consider the tower $\T_{m\cdot (\abs{\F}+1)}$. Take $G=\{g_i^{m\cdot (\abs{\F}+1)}:i\leq  \abs{\F} \}\subseteq{\T_{m\cdot (\abs{\F}+1)}}$, i.e., a set of the first $\abs{\F}+1$ partial functions of $\T_{m\cdot (\abs{\F}+1)}$. Now, we use the standard pigeonhole principle. We have $m\cdot(\abs{\F}+1)$ points in the domain of $G$ and for each $l\in\dom{G}$ the set $\{g_i^{m\cdot(\abs{\F}+1)}(l):i\leq \abs{\F}\}$ is of cardinality $\abs{\F}+1$, so there is $i(l)\leq \abs{\F}$ such that $\ev{l,g_{i(l)}^{m\cdot(\abs{\F}+1)}(l)}\notin\bigcup\F$. The set $\{\ev{l,g_{i(l)}^{m\cdot(\abs{\F}+1)}(l)}: l\in\dom{G} \}$ contains $m\cdot(\abs{\F}+1)$ points, but there are only $\abs{\F}+1$ possible values of $i(l)$. By pigeonhole principle, one of these numbers $i(l)$ must have been repeated at least $m$ times. Denote this number by $i_m$. 
    
    Thus, for any $m\geq 1$ there is $i_m\leq\abs{\F}$ such that $\abs{g_{i_m}^{m\cdot(\abs{\F}+1)}\setminus\bigcup\F}\geq m$. A sequence $\ev{i_m:m\in\omega}$ takes only finitely many values, thus, at least one of these values, say $i$, must be repeated infinitely many times. That is, $(\exists^\infty m\in\omega)\ \abs{g_{i}^{m\cdot(\abs{\F}+1)}\setminus \bigcup \F}\geq m$.
\end{proof}
Consequently, using the notation from Lemma~\ref{kkveze}, for any $A\in\sel$ such that $A\subseteq\bigcup\F$ we have that $(\forall m\in\omega)(\exists n\geq m)\ \abs{g_i^n\setminus A}\geq m$
holds as well.

There are ways how to express the ideals $\selp{\A}$ and $\EDp{\A}$ similar to those in Section~\ref{natural_finfin}. These definitions are probably better in expressing the intuition behind these ideals.
\begin{center}
        \begin{tabular}{lcl}
             $\selp{\A}$&$=$ &$\{X\subseteq\omega\times\omega: (\exists m\in\omega)(\forall A\in \A)\ \abs{X\cap A}\leq m \}$,\\[0.2cm]
             $\EDp{\A}$&$=$ & $\{X\subseteq\omega\times\omega: (\exists m\in\omega)\ \{A\in\A: \abs{X\cap A }> m \}\text{ is finite}\}$.
        \end{tabular}
\end{center}

\begin{prop}\label{sufficient}
\
\begin{enumerate}[\rm (A)]
    \item If $\sel$ is a $\pid{\finop{\B}}$-ideal then there is no sequence $\seqn{\T_k}{k}$ of essentially different towers with each~$\T_k$ being a~$\tow{1}{k}$ w.r.t.~$\B$.\footnote{In fact, we shall prove the~equivalence, see Theorem~\ref{ref1}.}
    \item If $\sel$ is a $\pid{\selp{\B}}$-ideal then there is no $\tow{k}{k}$ w.r.t. $\B$ for some $k\in\omega$.
    \item If $\sel$ is a $\pid{\EDp{\B}}$-ideal then there is no sequence $\langle \T_k:k\in\omega \rangle$ of essentially different towers with each~$\T_k$ being a~$\tow{k}{k}$ w.r.t. $\B$.
    \item If $\sel$ is a $\pid{\ofinp{\B}}$-ideal then there is no $\tow{\omega}{k}$ w.r.t. $\B$ for some $k\in\omega$.\footnote{In fact, we shall prove the~equivalence, see Theorem~\ref{veze}.}
    \item If $\sel$ is a $\pid{\finfinp{\B}}$-ideal then there is no sequence $\langle \T_k:k\in\omega \rangle$ of essentially different towers with each~$\T_k$ being an~$\tow{\omega}{k}$ w.r.t. $\B$.
\end{enumerate}
\end{prop}
\proof
(A) Let $\ev{\T_k:k\in\omega}$ be a sequence of essentially different towers with each~$\T_k=\{g^k_i:i<k \}$ being a~$\tow{1}{k}$ w.r.t.~$\B$. Define a partial function $f_i=\bigcup_{k>i}g^k_i$ for each $i\in\omega$. We will show that $\W=\{f_i:i\in\omega \}\subseteq\sel$ is a~witness for $\sel$ not being a~$\pid{\finop{\B}}$-ideal. First, note that $f_i\cap f_j=\emptyset$ for $i\neq j$ and $\dom{f_i}=^*\dom{f_j}$ for each $i,j<\omega$. Let wlog $X=\bigcup \F\in\sel$ where $\F$ is a finite subset of ${^\omega\omega}$. Pick a~set $\W'$ of $\abs{\F}+1$ distinct functions from $\W$. Since we have $\abs{\F}+1$ disjoint partial functions from $\sel$ on $\bigcap_{f\in\W'}\dom{f}$, i.e.,  $\tower{\omega}{\abs{\F}+1}$, by Lemma~\ref{omegaveze} there exists $f'\in\W'$ such that $\abs{f'\setminus X}=\omega$. However, $\W$ was chosen in such a~way that infinite parts of its elements do not belong to~$\finop{\B}$, so $f'\setminus X\notin\finop{\B}$. Thus, $\sel$ is not a $\pid{\finop{\B}}$-ideal.
\par
(B) Assume that there is a $\tow{k}{k}$ w.r.t. $\B$ for all $k\in\omega$. Note that if there is a $\tow{k}{k}$ for each $k\in\omega$, then there is a sequence $\langle \T_k\subseteq\mathcal{P}(\omega\times\omega):k\in\omega\rangle$ such that
\begin{itemize}
    \item[\rm (a)] $\T_k=\{g^k_i:i<k\}$ is a $\tow{k}{k}$ for each $k<\omega$,
    \item[\rm (b)] $\dom{\T_i}\cap\dom{\T_j}=\emptyset$ for each $i\neq j$.
\end{itemize}

Let us define $f_i=\bigcup_{k>i}g^k_i$ for every $i<\omega$ and $\W=\{f_i:i<\omega \}$. Let $X\in\sel$, wlog $X=\bigcup\F$ for some finite $\F\subseteq {^\omega\omega}$. Applying Lemma~\ref{kkveze}, we get $i\in\omega$ such that $(\forall m\in\omega)(\exists n\geq m)\ \abs{g_i^n\setminus\bigcup\F}\geq m$. Since $g_i^n$ is a monochromatic part of $f_i$ for all but finitely many $n$, by the definition of $\selp{\B}$ we get $f_i\setminus X\notin\selp{\B}$. Note that it does not matter if the color changes with $n$.
\par
(C)  Assume there is a sequence $\T_k=\{g^k_i:i<k \}$ of esentially different towers, each $\T_k$ being a $\tow{k}{k}$ w.r.t. $\B$. Let us define $f_i$'s and $\W$ as in the previous cases and let, again, wlog $X=\bigcup\F$ for some finite $F\subseteq {^\omega\omega}$. The rest of the argumentation follows the lines of the previous case with a single exception: $g_i^n$ and $g_i^m$ are of different colors for any $n\neq m$. This fact is crucial for $f_i\setminus X\notin \EDp{\B}$.
\par
(D)  Assume that there is an~$\tow{\omega}{k}$,  $\T_k=\{g_i^k: i<k \}$ w.r.t. $\B$, for any $k\in\omega$ and consider a countable family $\{g_i^k:i,k\in\omega \}$. Let $X=\bigcup\F\in\sel$ for some finite $\F\subseteq{^\omega\omega}$. Applying Lemma~\ref{omegaveze} we get $i\leq \abs{\F}$ such that $\abs{g_i^{\abs{\F}+1}\setminus\bigcup\F}=\omega$.
\par
(E) Assume there is a sequence $\T_k=\{g^k_i:i<k \}$ of essentially different towers, each $\T_k$ being an~$\tow{\omega}{k}$ w.r.t. $\B$. Define $f_i$'s as in the cases (A), (B), (C) and let $X=\bigcup\F\in\sel$ for some finite $\F\subseteq{^\omega\omega}$. Let us denote by $\overline{\T}_k=\{g_i^k:i<\abs{\F}+1\}\subseteq \T_k$ for any $k>\abs{\F}$. Clearly, $\overline{\T}_k$ is an~$\tow{\omega}{\abs{\F}+1}$ for any $k>\abs{\F}$. Hence, by Lemma~\ref{omegaveze} there is a sequence $\ev{i_k:k\in\omega}\in{^\omega\{0,\dotso,\abs{\F} \} }$ such that $\abs{g^k_{i_k}\setminus\bigcup \F}=\omega$ for every $k>\abs{\F}$. Clearly, one of $i_k$'s, say $i'$, must have been repeated infinitely many times, i.e., $(\exists^\infty k\in\omega)\ \abs{g_{i'}^k\setminus\bigcup\F}=\omega$.Thus, $f_{i'}\setminus\bigcup\F\notin\finfinp{\B}$.
\qed

\section{Eventually different ideal and towers are sometimes enough}\label{towers_equiv}

As we have already mentioned, towers of functions in some cases describe not only a necessary condition, but also a sufficient condition for the~ideal $\sel$ to be a~$\mathrm{P}(\J)$-ideal. We shall try to point out these cases and show that sometimes they describe that the~ideal $\ED$ is a~$\mathrm{P}(\J)$-ideal as well. We begin with an~auxiliary assertion. 

Let us recall the fact that $\EDp{\A}\not\subseteq\finop{\B}$ for any partitions $\A,\B$ because of the tallness of $\EDp{\A}$. However, by Lemma~\ref{existujeInkIzo} there are partitions~$\A,\B$ such that $\selp{\A}\subseteq\finop{\B}$.

\begin{theorem}\label{ref1}
The following statements are equivalent.
\begin{itemize}
    \item[\rm (1)] $\sel$ is a $\pid{\finop{\B}}$-ideal.
    \item[\rm (2)] $\ED$ is a $\pid{\finop{\B}}$-ideal.
    \item[\rm (3)] $\ED\perp\finop{\B}$.
    \item[\rm (4)] $\rsub{\ED}{\finop{\B}}$.
    \item[\rm (5)] $\rsub{\sel}{\finop{\B}}$.
    \item[\rm (6)] $\sel,\fino$ are $\pid{\finop{\B}}$-ideal.
    \item[\rm (7)] $\omega\times\omega$ can be covered by finitely many verticals, ${B_n}$'s, and functions in~$\baire{\omega}$.
    \item[\rm (8)] There is no sequence $\seqn{\T_k}{k}$ of essentially different towers with each~$\T_k$ being a~$\tow{1}{k}$ w.r.t.~$\B$.
\end{itemize}
\end{theorem}
\proof (2)$\equiv$(6), and (2)$\to$(1) follow directly from Proposition~\ref{differJ}. Part (7) is just (3) restated in terms of generating families. Implications (3)$\to$(4) and (5)$\to$(1) follow directly from Proposition~\ref{ortaink}. Part (3)$\to$(2) follows from Lemma~\ref{ort}. The~implication (1)$\to$(8) is shown in Proposition~\ref{sufficient}.
\par
(4)$\to$(5) Let $F\in\ED^*$ be a witness for (4). If $F\in\sel^*$ we are done. If $(\omega\times\omega)\setminus F$ contains infinite sections of ${V_n}$'s it is enough to add this sections to $F$ and form $F'$. It is not difficult to see that $F'\in\sel^*$ and since $\sel\restriction V_n=[V_n]^{<\omega}$ for each $V_n$, if $E\in\sel\restriction F'$ then $E$ is just union of a set belonging to $\sel\restriction F\subseteq\finop{\B}$, and some finite set from $\sel\restriction \bigcup_{n\in G} V_n$ for some finite $G\subseteq\omega$.
\par
$\neg$(3)$\to\neg$(8) Let us assume that $\ED$ and $\finop{\B}$ are not orthogonal. We inductively construct a~sequence $\seqn{\T_k}{k}$ of essentially different towers such that each~$\T_k$ is a~$\tow{1}{k}$ w.r.t.~$\B$.  Pick $\T_1$ arbitrarily contained in $V_0\cap B_{k_1}$  for some $k_1\in\omega$. Set $m_1=0$. For each $n\geq 2$, let 
$$
X_n=\bigcup_{k>k_{n-1}}B_k.
$$
We have $X_n\notin\ED$ since $\ED$ and $\finop{\B}$ are not orthogonal. Consequently, $\limsup_{m\to\infty}\abs{V_m \cap X_n}=\infty$, so there is an~$m_n>m_{n-1}$ such that $\abs{V_{m_n}\cap X_n}\geq n+1$. Pick $\T_n\subseteq V_{m_n}\cap X_n$ of size~$n+1$ arbitrarily, and set $k_n>k_{n-1}$ to be a~number with $\T_n\subseteq\bigcup_{k\leq k_n}B_k$. It is straightforward that the~construction leads to the~required sequence $\seqn{\T_k}{k}$ of essentially different towers.
\qed

One can notice that the~condition $\sel\perp\finop{\B}$ is not listed in~Theorem~\ref{ref1}. In the following we shall provide the reason why it cannot be added to the list.

\begin{lema}\label{selroz}
If $\sel\perp\finop{\B}$ then $\fino\subseteq\finop{\B}$.
\end{lema}
\proof It suffices to notice that every $E\in\sel$ has only finite intersection with each $V_n$. If $E\in\sel$ is such that $(\omega\times\omega)\setminus E\in \finop{\B}$ then $V_n\setminus E\subseteq(\omega\times\omega)\setminus E\in \finop{\B}$. The rest of the proof is implied by the fact $V_n\setminus E=^* V_n$, so we get $V_n\in\finop{\B}$ for all $n\in\omega$. \qed

\begin{prop}
If $\sel\perp \finop{\B}$ then $\ED$ is a $\pid{\finop{\B}}$-ideal. The converse is not generally true.
\end{prop}
\proof
It follows from Lemma~\ref{selroz} that $\fino$ is a $\pid{\finop{\B}}$-ideal. Ideal $\sel$ is also a $\pid{\finop{\B}}$-ideal because of the orthogonality. Now it suffices to apply Proposition~\ref{differJ}. For the second part let $\B$ be an arbitrary infinite partition into infinite sets such that: $B_0=\bigcup_{n>0}(\{n\}\times\omega)$, $\bigcup_{n>0}B_n=\{0\}\times \omega$. It is easy to see that $\ED\perp \finop{\B}$ but $\sel\subseteq \finop{\B}$.   \qed

Another consequence of Theorem~\ref{ref1} states that if $\ED$ is a $\pid{\finop{\B}}$-ideal then $\finop{\B}$ is a $\pid{\ED}$-ideal. Note that the~opposite implication is not true in general, since $\fino\subseteq\ED$. 
\par
Towers of functions can be used to characterize  ``$\sel$ is a~$\pid{\ofinp{\B}}$-ideal'' as well.

\begin{theorem}\label{veze}
The following statements are equivalent.
\begin{itemize}
    \item[\rm (1)] $\sel$ is a $\pid{\ofinp{\B}}$.
    \item[\rm (2)] $\rsub{\sel}{\ofinp{\B}}$.
    \item[\rm (3)] There is $k\in\omega$ such that there is no $\tow{\omega}{m}$ w.r.t. $\B$ for every $m>k$.
    \item[\rm (4)] $(\forall\mathcal{F}\in[{^\omega\omega}]^\omega)(\exists F\in[{^\omega\omega}]^{<\omega})(\forall f\in\mathcal{F})(\forall B\in\B)\ \abs{(f\cap B)\setminus \bigcup F}<\omega$.
\end{itemize}
\end{theorem}
\proof
    $\neg$(3)$\to\neg$(1) This is, in fact, Proposition~\ref{sufficient}. Note that if there is a $k$ such that there is no $\tow{\omega}{k}$ w.r.t. $\B$, then there is no $\tow{\omega}{m}$ w.r.t. $\B$ for every $m>k$.

(3)$\to$(2) For each (color) $B\in \B$ let $$X_B^i=\{n:\abs{V_n\cap B}=i\},\ \ i\in\omega+1\setminus\{ 0\},$$
and $I_B=\{i\leq k: \abs{X_B^i}=\omega \}$. Note that we may omit the case $i>k$ since $X^i_B$ would be always finite by the assumption. If $i\in I_B$ then there is an $\tow{\omega}{i}$ $\{g^1_{B,i},\dotso, g^i_{B,i}\}$ such that $\{\langle n,m \rangle\in B: n\in X_B^i  \}=\bigcup_{j=1}^ig^j_{B,i}$. Let us define a countable family $\G=\{g^j_{B,i}:B\in\B, i\in I_B, 1\leq j\leq i \}$ and consider an arbitrary bijective enumeration of $\G=\{g_j:j<\omega \}$. The family $\G$ consists of mutually disjoint monochromatic functions. We will show that there is a set of at most $k$ functions $f_j\in{^\omega\omega}$ such that $g_i\setminus\bigcup_{j<k}f_j$ is finite for each $i$.

Define \begin{alignat*}{2}
& h_i=g_i, \quad  && \text{ for }i<k,\\
& h_i=g_i \setminus \left(g_i\restriction\bigcup_{\{m_1,\dotso, m_k \}\in[i]^k}\left(\bigcap_{j=1}^k \dom{g_{m_j}}\cap\dom{g_i} \right)\right),\quad && \text{ for } i\geq k. 
\end{alignat*}
Notice that the set $\bigcap_{j=1}^k \dom{g_{m_j}}\cap\dom{g_i}$ is always finite: if $\bigcap_{j=1}^k \dom{g_{m_j}}$ is finite we are done. If $\bigcap_{j=1}^k \dom{g_{m_j}}$ is infinite, then $\bigcap_{j=1}^k \dom{g_{m_j}}\cap\dom{g_i}$ is an intersection of the domain of $g_i$ with the domain of an $\tow{\omega}{k}$. Since there does not exist an~$\tow{\omega}{k+1}$, it must be finite. As a consequence, a finite union $$\bigcup_{\{m_1,\dotso, m_k \}\in[i]^k}\left(\bigcap_{j=1}^k \dom{g_{m_j}}\cap\dom{g_i} \right)$$ is also a finite set, thus $h_i=^*g_i$. We show by contradiction that the set $H=\bigcup_{i<\omega} h_i$ is a union of at most $k$ partial functions $f_j$, i.e., $\abs{V_n\cap H}\leq k$ for all $n\in\omega$. If $\abs{V_n\cap H}>k$ for some $n$, then there are distinct $i_1,\dotso, i_{k+1}$ such that $n\in \dom{h_{i_1}},\dotso,\dom{h_{i_{k+1}}}$ and consequently $n\in \dom{g_{i_1}},\dotso,\dom{g_{i_{k+1}}}$. Let wlog $i_{k+1}=\max\{i_1,\dotso,i_{k+1} \}$. Then $\{i_1,\dotso,i_k \}\in[i_{k+1}]^k$ and $n\in\bigcap_{j=1}^k\dom{g_{i_j}}\cap \dom{g_{i_{k+1}}}$, thus $n\notin\dom{h_{i_{k+1}}}$ by the definition of $h_{i_{k+1}}$, a contradiction. Moreover, $g_i\setminus H$ is finite since $g_i=^*h_i$ for each $i\in\omega$.

Now, it suffices to show that $H\in\sel$ is a witness for $\rsub{\sel}{\ofinp{\B}}$. We will show that $f\setminus H\in\ofinp{\B}$ for every function $f\in{^\omega\omega}$, which is a subbase of $\sel$, consequently the relation $\rsub{\sel}{\ofinp{\B}}$ holds true. By the definition, $f\setminus H\in\ofinp{\B}$ if and only if $(f\cap B)\setminus H$ is finite for every $B\in\B$. First, note that not only $X_B^i$ is finite for every $i>k$ but also $\bigcup_{i\notin I_B}X_B^i=\bigcup_{i\in \omega+1\setminus (I_B\cup\{0 \})}X_B^i\in \fin$, otherwise there would exist an~$\tow{\omega}{k+1}$ (there would be infinitely many columns with at least $k+1$ points in~$B$). Denote by
$$C_B=\bigcup_{i\in I_B}\{\langle n,m \rangle\in B:n\in X_B^i \},\ \ D_B=B\setminus C_B.$$
Then $B=C_B\cup D_B$. 

We know that $\{\langle n,m \rangle\in B: n\in X_B^i  \}=\bigcup_{j=1}^ig^j_{B,i}$ and $I_B$ is finite, thus $C_B$ is finite union of monochromatic towers, i.e., there are $m<\omega$ and $i_0,\dotso, i_{m-1}<\omega$ such that $C_B= \bigcup_{j<m} g_{i_j}$, so $C_B\setminus H\subseteq \bigcup_{j<m} g_{i_j}\setminus H$ is a finite set. Clearly, $\mathrm{proj}_1 D_B=\bigcup_{i\in \omega+1\setminus (I_B\cup\{0 \})}X_B^i$, which is a finite set, then $f\cap D_B$ is finite because $\dom{f}\cap \mathrm{proj}_1 D_B\in\fin$. We have
$$(f\cap B)\setminus H\subseteq (C_B\setminus H) \cup (f\cap D_B)\in \fin.$$
is finite.

(2)$\to$(1) Proposition~\ref{ortaink}(b).

(1)$\equiv$(4) Assume (1) and let $\F\in[{^\omega\omega}]^{\omega}$. Then $\F\subseteq\sel$ and by the assumption there is $I_\F\in \sel$ such that $f\setminus I_\F\in\ofinp{\B}$ for every $f\in \F$. By the definition of $\sel$ there is a finite set of functions $F\subseteq {^\omega\omega}$ with $I_\F\subseteq\bigcup F$, so we have $f\setminus \bigcup F\in\ofinp{\B}$ for any $f\in \F$ as well. By the definition of $\ofinp{\B}$ the $f\setminus \bigcup F\in\ofinp{\B}$ is equivalent to $(\forall B\in \B)\ \abs{(f\cap B)\setminus\bigcup F}<\omega$.  For the other implication let $\{I_n:n\in\omega \}\subseteq \sel$. By the definition of $\sel$, for each $n$ there is $m_n$ and $f^n_1,\dotso,f^n_{m_n}$ such that wlog $I_n=\bigcup_{j<m_n}f^n_j$. Set $\mathcal{F}=\{f^n_j:n,j<\omega \}$ and apply (4). Then $I_n\setminus\bigcup F=\bigcup_{j<m_n}\left(f^n_j\setminus\bigcup F\right)$. For $j<m_n$ we know that $(f^n_j\cap B)\setminus\bigcup F$ is finite for each $B\in \B$, i.e., $f^n_j\setminus\bigcup F\in\ofinp{\B}$. Thus, $I_n\setminus\bigcup F$ is finite union of sets from $\ofinp{\B}$, thus, a set from $\ofinp{\B}$.
\qed

Although the ideal $\ED$ was added to the~list of equivalent conditions in Theorem~\ref{ref1}, Theorem~\ref{veze} is not the case as one can see in Table~\ref{omxom}. However, $\ED$ is a $\pid{\ofinp{\B}}$-ideal if and only if these ideals are orthogonal, as one can see below. 

\begin{theorem}\label{ed_ofin}
The following statements are equivalent.
\begin{itemize}
    \item [\rm (1)] $\ED$ is a $\pid{\ofinp{\B}}$-ideal.
    \item[\rm (2)] $\rsub{\ED}{\ofinp{\B}}$.
     \item[\rm (3)] $\ED\perp\ofinp{\B}$.
    \item[\rm (4)] $\fino$ and $\sel$ are $\pid{\ofinp{\B}}$-ideals.
    \item[\rm (5)] $\rsub{\fino}{\ofinp{\B}}$ and $\rsub{\sel}{\ofinp{\B}}$.
\end{itemize}
\end{theorem}
\proof 
(1)$\to$(4) follows from Proposition~\ref{differJ}. (4)$\to$(5) follows from Theorem~\ref{veze} and Corollary~\ref{prehlcor}. 

For (5)$\to$(2) it is sufficient to note the following -- it can be easily proved that if $\rsub{\I_1}{\J}$ and $\rsub{\I_2}{\J}$ where $\I_1\not\perp\I_2$, then $\rsub{\I_1\lor\I_2}{\J}$. 

(2)$\equiv$(3) is a consequence of Lemma~\ref{nowheretall}.

(3)$\to$(1) follows from Lemma~\ref{ort}.
\qed

At this point it may seem that $\rsub{\I}{\J}$ and $\I$ is a $\pid{\J}$-ideal are always equivalent but this, in fact, is not the case. Consider, e.g., $\ofin$ and $\fino$. Clearly, $\ofin\not\subseteq^\upharpoonright \fino$ but $\ofin$ is a $\pid{\fino}$-ideal.

\section{Towers of functions are not enough}

In contrary to Section~\ref{towers_equiv}, towers of functions are not always sufficient to fully describe the~$\pid{\J}$-property for $\sel$. The~whole section is devoted to the~counterexample, its~construction, and a~proof that it possesses the~required properties. 
\par
We shall define a~partition $\E=\set{A_{j,i}}{i,j\in\omega}\cup \{B_i:i\in \omega \}$ of~$\omega\times\omega$ into infinite sets.
Let us fix an~auxiliary partition~$\D=\set{D_k}{k\in\omega}$ of~$\omega\setminus \{0 \}$ into infinite sets. Next, we set
\begin{itemize}
    \item $\langle m,i\rangle \in A_{j,i}$ if and only if $m$ is the~$(j+ki)$-th element of~$D_k$, 
    \item $\bigcup \{B_i:i\in\omega\} \supseteq \{0 \}\times\omega$,
    \item the~elements of the~remaining subset of~$\omega\times\omega$, say $\{w_i:i\in\omega \}$, are distributed one-by-one into sets~$B_i$, i.e., $w_i\in B_i$ for all $i\in\omega$.
\end{itemize}
The last condition implies that if $m$ is $(j+ki)$-th element of $D_k$ for $-ki\leq j<0$ then $\ev{m,i}\in \bigcup \{B_i:i\in\omega \}$.
If a point $\ev{x,y}$ is colored with $A_{j,i}$, then $y=i$ because in the definition of $\E$ the color $A_{j,i}$ is assigned only to points whose second coordinate equals to $i$. The parameter $i$ in the notation $A_{j,i}$ thus represents the number of the row in which the color $A_{j,i}$ occurs. Consequently, $A_{j,i}\cap A_{j',i'}=\emptyset$ whenever $i\neq i'$. Also note that $\abs{A_{j,i}\cap (D_k\times\{ i\})}=1$ for all $i,j,k\in\omega$, that is, there is only one point of $\omega\times\omega$ ``above'' $D_k$ colored with $A_{j,i}$.
\par
The~construction of coloring~$\E$ is done in a~way to avoid towers of functions. 
\par\vspace{0.3cm}
\noindent \textbf{Observation.} There is no $\tow{k}{k}$ w.r.t. $\E$ for $k\geq 3$.
\proof
 By the definition of $\E$ for any $i\in\omega$ there is exactly one element $\ev{x,y}$ of $\omega\times\omega$ with $x>0$ such that $\ev{x,y}\in B_i$. Hence, if $f$ is a function with $f\subseteq B_i$ for some $i\in\omega$, then $\abs{\dom{f}}\leq 2$. Thus, functions colored with colors from $\{B_i:i\in\omega \}$ cannot be elements of $\tow{3}{3}$s.

As a result, it is enough to consider the colors $\{A_{j,i}:i,j\in\omega \}$ and the set $(\omega\times\omega)\setminus (\{ 0\}\times\omega)$. With this setup, we will show that there is not even $\tow{2}{2}$, which implies non-existence of $\tow{3}{3}$. Let $m_0, m_1\in\omega\setminus\{ 0\}, m_0\neq m_1,i_0,i_1\in\omega$ be such that   $\ev{m_0,i_0},\ev{m_1,i_1}\in\omega\times\omega$ have the same color, i.e., there are $j,i\in\omega$ such that $\ev{m_0,i_0},\ev{m_1,i_1}\in A_{j,i}$. As mentioned previously, $A_{j,i}$ occurs only in the row $\omega\times \{ i\}$, thus $i_0=i_1=i$. So we have $\ev{m_0,i},\ev{m_1,i}\in A_{j,i}$. 

Now, note that there are $k_0\neq k_1$ such that $m_0\in D_{k_0}$ and $m_1\in D_{k_1}$. If we assumed, by way of contradiction, that there is $k\in\omega$ such that $m_0, m_1\in D_k$, then $m_0$ is $(j+ki)$-th element of $D_k$ and so is $m_1$, a contradiction because $m_0\neq m_1$.

By the definition $m_0$ is $(j+k_0i)$-th element of $D_{k_0}$ and  $m_1$ is $(j+k_1i)$-th element of $D_{k_1}$. Consider now $\ev{m_0,i'},\ev{m_1,i'}$ for some $i'\neq i$ and assume they are colored with the same color $A_{j',i'}$. Then
$$m_0\text{ is }(j'+k_0i')\text{-th element of }D_{k_0},$$
$$m_1\text{ is }(j'+k_1i')\text{-th element of }D_{k_1}.$$
So we have that $m_0$ is $(j+k_0i)$-th and also $(j'+k_0i')$-th element of $D_{k_0}$,  $m_1$ is $(j+k_1i)$-th and also $(j'+k_1i')$-th element of $D_{k_1}$, i.e.,
$$j'+k_0i'=j+k_0i\ \land\ j'+k_1i'=j+k_1i$$
which is equivalent to
$$i'(k_0-k_1)=i(k_0-k_1)\ \equiv\ i=i',$$
a contradiction. \qed

We will show that $\sel$ is not a~$\pid{\selp{\E}}$-ideal, i.e., the necessary condition stated in Proposition~\ref{sufficient} in terms of existence of towers of monochromatic functions is not a sufficient condition.
\begin{theorem}\label{tazkaThm}
$\sel$ is not a $\pid{\selp{\E}}$-ideal.
 \end{theorem}
\proof By contradiction. Assume $\sel$ is $\pid{\selp{\E}}$-ideal. Set $f_i=\omega\times\{ i\}$. We will show that $\W=\{f_i:i\in\omega \}$ is a witness for $\sel$ being not a $\pid{\selp{\E}}$-ideal.

Since $\sel$ is a $\pid{\selp{\E}}$-ideal, there is $X\in\sel$, $X\subseteq \bigcup\F$ for some $\F\subseteq {^\omega\omega},\abs{\F}<\omega$, such that $f_i\setminus X\in\selp{\E}$ for all $i\in\omega$, i.e., there is $\langle k_i\in\omega:i<\omega \rangle$ with $\abs{(f_i\setminus X)\cap A}\leq k_i$ for all $A\in\E$ and $i\in\omega$. 

\vspace{0.3cm}
Now, we introduce some new notation and definitions for the~sake of comprehensibility of the~proof. We define a~function assigning a~color lying  below a~selected color on a set~$D_k$, i.e., function $(\cdot\downarrow\cdot)$, by setting 
$$(A_{j,i}\downarrow D_k)=A_{j',i-1}$$
if $\proj{1}{A_{j,i}}\cap D_k=\proj{1}{A_{j',i-1}}\cap D_k$.

\begin{claim}\label{claim1}
$(A_{j,i+1}\downarrow D_k) = A_{j+k,i}$ for all $i,j,k\in\omega$.
\end{claim}
\proofofclaim Let $\ev{m,i+1}\in A_{j,i+1}$ with $m\in D_k$, i.e., $m$ is $(j+k(i+1))$-th element of $D_k$. One can easily see there is a $j'\in\omega$ such that $(A_{j,i+1}\downarrow D_k)=A_{j',i}$. By the definition, $m$ is also $(j'+ki)$-th element of $D_k$. We have that $$j+k(i+1)=j'+ki\ \equiv\ j'=j+k.$$ \qed

An~integer interval $\{x\in\omega: a\leq x\leq b \}$ with endpoints $a,b\in\omega$ is denoted by $\boldsymbol{[a,b]}$. If $s\leq t,u\leq v$ then the~set 
$$\bigcup_{j=s}^t A_{j,i}\cap \bigcup_{k=u}^v\left(D_k\times\{ i\}\right)$$
is called an~$\boldsymbol{\chain{i}{s,t}{u,v}}$. Note that an~$\chain{i}{s,t}{u,v}$ is a~finite function that can be partitioned into $d(a)\coloneqq\abs{[u,v]}=v-u+1$ many functions $g_j$, $j<d(a)$, such that $g_j$ is defined on $D_{u+j}$ and each $g_j$ contains $l(a)\coloneqq\abs{[s,t]}=t-s+1$ many elements.

\begin{claim}\label{claim2}
For every $\chain{i+1}{s,t}{u,v}$ $a$ there is an $\chain{i}{s',t'}{u,v}$ $b$ with $l(b)=l(a)-d(a)+1$ and $\dom{b}\subseteq \dom{a}$.
\end{claim}
\proofofclaim Denote by $L=l(a)-d(a)+1$. Consider intervals $[(t-L+1)-n,t-n]$ and $D_{u+n}$ for $n<d(a)$. 

Note that for any $n<d(a)$
$$a_n\coloneqq\bigcup_{\substack{j=(t-L+1)-n \\ =t-n-(L-1)}}^{t-n}A_{j,i+1}\cap (D_{u+n}\times\{i+1\})\subseteq a,$$
because $t-n\leq t$ and $(t-L+1)-n>(t-L+1)-d(a)=(t-(l(a)-d(a)+1)+1)-d(a)=t-l(a)=t-(t-s+1)=s-1$, so $(t-L+1)-n\geq s$. Also, Claim~1 gives us $$(A_{t-n-o,i+1}\downarrow D_{u+n})=A_{t+u-o,i}\ \text{for all }o<L.$$

We have that
$$\bigcup_{j=t-n-(L-1)}^{t-n} (A_{j,i+1}\downarrow D_{u+n})\cap (D_{u+n}\times\{ i\})=\bigcup_{o=0}^{L-1}(A_{t-n-o,i+1}\downarrow D_{u+n})\cap (D_{u+n}\times\{ i\})=$$
$$=\bigcup_{o=0}^{L-1}A_{t+u-o,i}\cap (D_{u+n}\times\{i \}).$$
By the definition of $(\cdot\downarrow\cdot)$ we have

\begin{equation}\tag{*}
\proj{1}{a_n}=\proj{1}{\bigcup_{o=0}^{L-1}A_{t+u-o,i}\cap (D_{u+n}\times\{i \})}. 
\end{equation}

Set $$b=\bigcup_{n<d(a)}\left( \bigcup_{o=0}^{L-1}A_{t+u-o,i}\cap (D_{u+n}\times\{i \})\right)=\bigcup_{o=0}^{L-1}A_{t+u-o,i}\cap\bigcup_{n=0}^{d(a)-1}(D_{u+n}\times\{i \})=$$
$$=\bigcup_{o=0}^{L-1}A_{t+u-o,i}\cap \bigcup_{k=u}^v(D_k\times\{ i\}).$$
The set $b$ is $\chain{i}{s',t'}{u,v}$ with $s'=t+u-(L-1), t'=t+u$, so $l(b)=t'-s'+1=L=l(a)-d(a)+1$.
Using disjointness of $D_k$'s, fact that $a_n\subseteq a$  for $n<d(a)$ and (*) we have that
$$\dom{b}=\proj{1}{\bigcup_{n<d(a)}\left( \bigcup_{o=0}^{L-1}A_{t+u-o,i}\cap (D_{u+n}\times\{i \})\right)}\subseteq \proj{1}{a}=\dom{a}.$$ \qed

The $\chain{i}{s,t}{u,v}$ $a$ is \textbf{$\boldsymbol{X}$-covered} if $a\subseteq X$. The color $A_{j,i}$ is \textbf{$\boldsymbol{X}$-covered on $\boldsymbol{[u,v]}$ (on $\boldsymbol{D_k})$} if $$A_{j,i}\cap\bigcup_{k=u}^v(D_k\times\{ i\})\subseteq X\ \ \  (A_{j,i}\cap (D_k\times\{ i\})\subseteq X \text{ respectively}).$$

\begin{claim}\label{claim3}
For every $\chain{i}{s',t'}{u,v}$ $b$ with $d(b)\geq (k_i+1)^{l(b)}$ there is an $X$-covered $\chain{i}{s',t'}{u',v'}$ $c$ such that $d(c)=\floor*{\frac{d(b)}{(k_i+1)^{l(b)}}}$ and $c\subseteq b$.
\end{claim}
Before we prove this claim, let us point out that $l(c)=l(b)$. Also, we will need the following lemmas.
\begin{lema}\label{pravdivost}
Let $[u,v]$ be an integer interval with $d\coloneqq \abs{[u,v]}$. If $\varphi(x)$ holds for all but (at most) $k$ points in $[u,v]$, then there is an interval $[u',v']\subseteq [u,v]$ with $\abs{[u',v']}=\floor*{\frac{d}{k+1}}$, such that $\varphi(x)$ holds on $[u',v']$.
\end{lema}
\proof There exist $(k+1)$ disjoint subintervals of $[u,v]$ with cardinality $\floor*{\frac{d}{k+1}}$. Since $\varphi(x)$ is not true for (at most) $k$ points in $[u,v]$, $\varphi(x)$ must be true on (at least) one of these $(k+1)$ intervals by the pigeonhole principle.\qed 
\begin{lema}[folklore]\label{celacast}
 If $a\in\mathbb{Z},b,c,\in\omega\setminus\{ 0\}$, then $\floor*{\frac{\floor*{\frac{a}{b}}}{c}}=\floor*{\frac{a}{bc}}$.\qed
\end{lema}
\proofofclaim By the assumptions $\abs{(f_i\cap A_{j,i})\setminus X}=\abs{((\omega\times\{ i\})\cap A_{j,i})\setminus X}\leq k_i$ for all $j\in\omega$. We will inductively show that there are intervals $[u_j, v_j], s'\leq j\leq t'$ such that 
\begin{itemize}
    \item $[u_{j+1},v_{j+1}]\subseteq [u_j,v_j]$ for $s'\leq j<t'$,
    \item $A_{j,i}$ is $X$-covered on $[u_j,v_j]$ (hence, $A_{j,i}$ is $X$-covered also on $[u_{j'},v_{j'}]$ for $j'>j$ and $A_{j,i}$ is $X$-covered on $[u_{t'},v_{t'}]$ for every $j\in[s',t']$),
    \item $\abs{[u_j, v_j]}=v_j-u_j+1=\floor*{\frac{d(b)}{(k_i+1)^{(j+1)-s'}}}$ for $j\in [s',t']$.
\end{itemize}

Induction:
\begin{itemize}
    \item $j=s':$ Let $\varphi(k)$ be ``$A_{j,i}$ is $X$-covered on $D_k$''. By the assumptions $\varphi(k)$ is true for all but (wlog) $k_i$ points in $[u,v]$. After application of Lemma~\ref{pravdivost} we get $[u_{s'},v_{s'}]\subseteq [u,v]$ with $\abs{[u_{s'},v_{s'}]}=v_{s'}-u_{s'}+1=\floor*{\frac{d(b)}{k_i+1}}$.
    \item $s'\leq  j<t':$ Let $[u_j, v_j]$ be such that $A_{j,i}$ is $X$-covered on $[u_j,v_j]$ and $\abs{[u_j,v_j]}=\floor*{\frac{d(b)}{(k_i+1)^{(j+1)-s'}}}$. Let  $\varphi(k)$ be ``$A_{j+1,i}$ is $X$-covered on $D_k$''. By the assumptions $\varphi(k)$ is true for all but (wlog) $k_i$ points in $[u_j,v_j]$. By Lemma~\ref{pravdivost} there is $[u_{j+1},v_{j+1}]\subseteq[u_j,v_j]$ such that $A_{j+1,i}$ is $X$-covered on $[u_{j+1},_{j+1}]$  and 
    $$\abs{[u_{j+1},_{j+1}]}=\floor*{\frac{\abs{[u_j,v_j]}}{k_i+1}}=\floor*{\frac{\floor*{\frac{d(b)}{(k_i+1)^{(j+1)-s'}}}}{k_i+1}}=\floor*{\frac{d(b)}{(k_i+1)^{((j+1)+1)-s'}}}.$$
    The last equation follows from Lemma~\ref{celacast}. 
    \end{itemize}
 
 Consequently, we get $[u_{t'},v_{t'}]$  such that $A_{t',i}$ is $X$-covered on $[u_{t'},v_{t'}]$ and 
 $$\abs{[u_{t'},v_{t'}]}=\floor*{\frac{d(b)}{(k_i+1)^{t'-s'+1}}}=\floor*{\frac{d(b)}{(k_i+1)^{l(b)}}}>0.$$
 
 Since the sequence $[u_j,v_j]$ is $\subseteq$-decreasing, $[u_{t'},v_{t'}]\subseteq [u_j, v_j]$ for $ j\in[s',t']$.
 
 If we denote by $u'=u_{t'}, v'=v_{t'}$ and define $$c=\bigcup_{j=s'}^{t'}A_{j,i}\cap \bigcup_{k=u'}^{v'}(D_k\times\{i\})\subseteq X,$$ we can write 
 $$d(c)=\abs{[u',v']}=\floor*{\frac{d(b)}{(k_i+1)^{t'-s'+1}}}=\floor*{\frac{d(b)}{(k_i+1)^{l(b)}}}.$$
One can easily check that $c\subseteq b$. This is implied by the fact that $[u',v']\subseteq [u,v]$ and $[s',t']$ remained the same.\qed

Now we can finish the proof of the theorem. Define the following sequences:
\begin{flalign*}
&q_0=1, p_0=1,\\
&q_l=q_{l-1}(k_{l-1}+1)^{p_{l-1}}, \\ 
&p_l=q_l+p_{l-1}-1=q_{l-1}(k_{l-1}+1)^{p_{l-1}}+p_{l-1}-1.
\end{flalign*}

We will show that there is an $X$-covered $\chain{i}{s_i,t_i}{u_i,v_i}$ $a_i$ with $l(a_i)=p_i$, $d(a_i)=q_i$, for each $\abs{\F}\geq i\geq 0$ such that $\dom{a_{i+1}}\subseteq\dom{a_i}$ for $0\leq i < \abs{\F}-1$.

Induction:
\begin{itemize}
    \item Find an $X$-covered $\chain{\abs{\F}}{s_{\abs{\F}}, t_{\abs{\F}}}{u_{\abs{\F}}, v_{\abs{\F}}}$ $a_{\abs{\F}}$ with $s_{\abs{\F}}, t_{\abs{\F}}, u_{\abs{\F}}, v_{\abs{\F}}$ such that
    $l(a_{\abs{\F}})=p_{\abs{\F}}$, $d(a_{\abs{\F}})=q_{\abs{\F}}$. Note that in every row $\omega\times\{ i\}$ there is always such a chain.\footnote{Take, e.g., $s_{\abs{\F}}=0, t_{\abs{\F}}=p_{\abs{\F}}-1$. For any $j\in[0,p_{\abs{\F}}-1]$ there is $m_j$ such that $A_{j,\abs{\F}}$ is $X$-covered on $D_m$ for $m>m_j$. Set $u_{\abs{\F}}=\max\{m_j:j\in[0,p_{\abs{\F}}-1]\}+1$ and $v_{\abs{\F}}=u_{\abs{\F}}+q_{\abs{\F}}-1$.}
    
\item $0<i\leq \abs{\F}:$ let $a_i$ be an $X$-covered $\chain{i}{s_i,t_i}{u_i,v_i}$ with $l(a_i)=p_i, d(a_i)=q_i$. Then by Claim~\ref{claim2} there is an $\chain{i-1}{s_{i-1},t_{i-1}}{u,v}$ $a'_{i-1}$ with $$l(a'_{i-1})=l(a_i)-d(a_i)+1=p_i-q_i+1=q_i+p_{i-1}-1-q_i+1=p_{i-1},$$
$$d(a'_{i-1})=d(a_i),\ \ \dom{a'_{i-1}}\subseteq\dom{a_i}.$$
Note also that $a'_{i-1}$ satisfies the condition $d(a'_{i-1})\geq(k_{i-1}+1)^{l(a'_{i-1})}$. Hence, Claim~\ref{claim3} implies the existence of an $\chain{i-1}{s_{i-1},t_{i-1}}{u_{i-1},v_{i-1}}$ $a_{i-1}$ such that
$$d(a_{i-1})=\floor*{\frac{d(a'_{i-1})}{(k_{i-1}+1)^{l(a'_{i-1})}}}=\floor*{\frac{d(a_{i})}{(k_{i-1}+1)^{p_{i-1}}}}=\floor*{\frac{q_i}{(k_{i-1}+1)^{p_{i-1}}}}=$$
$$=\floor*{\frac{q_{i-1}(k_{i-1}+1)^{p_{i-1}}}{(k_{i-1}+1)^{p_{i-1}}}}=q_{i-1},$$
$$l(a_{i-1})=l(a'_{i-1})=p_{i-1}\ \ \text{and}\ \ a_{i-1}\subseteq a'_{i-1},$$
consequently $\dom{a_{i-1}}\subseteq \dom{a_i}$.
\end{itemize}

Lastly, we get an $X$-covered $\chain{0}{s_0,t_0}{u_0,v_0}$ $a_0$ with $d(a_0)=q_0=1$, $l(a_0)=p_0=1$ and $\dom{a_0}\subseteq \dom{a_j}$ for $0\leq j\leq \abs{\F}$, i.e., we get $m\in\dom{a_0}$ such that $\ev{m,i}\in X$ for $0\leq i\leq \abs{\F}$ and this cannot be done using $\abs{\F}$ functions, which X consists of, a contradiction.
\qed

It is possible to prove a~stronger result stated in Proposition~\ref{selnotp-ex_towers}. However, the~proof is a~more complex copy of the~proof of Theorem~\ref{tazkaThm}. We do not proceed with a~complete proof but rather point out which parts of the~proof of Theorem~\ref{tazkaThm} should be modified. 
 
\begin{prop}\label{selnotp-ex_towers}
$\sel$ is not a $\pid{\EDp{\E}}$-ideal.
\end{prop}
\proof By contradiction. Suppose $\sel$ is a $\pid{\EDp{\E}}$-ideal. Set again $f_i=\omega\times\{i\}$. We will show, that the family $\W=\{f_i:i\in\omega \}$ is a witness for $\sel$ not being a $
\pid{\EDp{\E}}$. By assumptions there is $X\in\sel$, $X\subseteq\bigcup\F$ for some $\F\subseteq {^\omega\omega}$, $\abs{\F}<\omega$ such that $f_i\setminus X\in\EDp{\E}$ for each $i\in\omega$, i.e., 
$$(\exists\langle k_i\in\omega:i<\omega \rangle\in{^\omega\omega})(\forall i\in\omega)(\forall^\infty A\in \E)\ \abs{(f_i\cap A)\setminus X}\leq k_i.$$
Since $\abs{f_i\cap B_j}\leq 2$ for any $i,j$ and $f_i\cap A_{j,i'}\neq \emptyset$ only for $i'=i$, it is equivalent to $$(\exists\langle k_i\in\omega:i<\omega \rangle\in{^\omega\omega})(\forall i\in\omega)(\forall^\infty j\in\omega)\ \abs{(f_i\cap A_{j,i})\setminus X}\leq k_i.$$ By the definition, $A_{j,i}\subseteq \omega\times\{ i\}=f_i$, hence $f_i\cap A_{j,i}=A_{j,i}$  for any $i,j\in\omega$. Thus, we get an equivalent formula
\begin{equation}\tag{*}
(\exists\langle k_i\in\omega:i<\omega \rangle\in{^\omega\omega})(\forall i \in \omega)(\forall^\infty j\in\omega)\ \abs{A_{j,i}\setminus X}\leq k_i.
\end{equation}
Let $C=\{j\in\omega: (\exists i\leq \abs{\F})\ \abs{A_{j,i}\setminus X}>k_i \}$. Clearly, $C\in\fin$ because for any $i\leq \abs{\F}$ we have that $\{j\in\omega: \abs{A_{j,i}\setminus X}>k_i \}\in\fin$ by (*).

We will use most of the apparatus built in the proof of Theorem~\ref{tazkaThm}, particularly Claim~\ref{claim1}, Claim~\ref{claim2}, modified Claim~\ref{claim3}, sequences $p_l$ and $q_l$ and the notation $a_i, a_i'$.
We will follow the lines of the proof of Theorem~\ref{tazkaThm}. In fact, the proof is identical. The only difference is in the modified version of Claim~\ref{claim3}, in which we have to carefully pick the first $X$-covered chain in the induction so as to avoid the ``bad'' colors (i.e., the colors $A_{j,i}$ such that $\abs{A_{j,i}\setminus X}>k_i$). 

Take an $X$-covered $\chain{\abs{\F}}{s_{\abs{\F}},t_{\abs{\F}}}{u_{\abs{\F}},v_{\abs{\F}}}$ $a_{\abs{\F}}$ with $s_{\abs{\F}},t_{\abs{\F}},u_{\abs{\F}},v_{\abs{\F}}$ being such that $l(a_{\abs{\F}})=p_{\abs{\F}}, d(a_{\abs{\F}})=q_{\abs{\F}}$ and $s_{\abs{\F}}>\max C$. It is easy to see that there is such a chain.\footnote{Set $s_{\abs{\F}}=\max C+1 >\max\{j\in\omega: \abs{A_{j,\abs{\F}}\setminus X}>k_{\abs{\F}} \}$ and $t_{\abs{\F}}=s_{\abs{\F}}+p_{\abs{\F}}-1$. Then for $j\in[s_{\abs{\F}},t_{\abs{\F}}]$ there is $m_j$ such that $A_{j,\abs{\F}}$ is $X$-covered on $D_m$ for $m>m_j$. Set $u_{\abs{\F}}=\max\{m_j: j\in [s_{\abs{\F}},t_{\abs{\F}}] \}+1$ and $v_{\abs{\F}}=u_{\abs{\F}}+q_{\abs{\F}}-1$.}

We modify the Claim~\ref{claim3} by adding the assumption $s'>\max C$, thus, the new formulation is as follows.
\begin{clm}{3*}\label{claim3*}
For every $\chain{i}{s',t'}{u,v}$ $b$  with $d(b)\geq (k_i+1)^{l(b)}$ and $s'>\max C$ there is an $X$-covered $\chain{i}{s',t'}{u',v'}$ $c$ such that $d(c)=\floor*{\frac{d(b)}{(k_i+1)^{l(b)}}}$ and $c\subseteq b$.
\end{clm}
The sole purpose of the condition $s'>\max C$ is to ensure that $\abs{A_{j,i}\setminus X}\leq k_i$ for $j\in[s',t']$. In fact, the assumption $s_{\abs{\F}}>\max C$ ensures $s_i>\max C$ for $i\leq\abs{\F}$. Hence, it is perfectly safe to apply Claim~\ref{claim3*} throughout the induction to outputs of Claim~\ref{claim2}.

The rest of the proof is identical.
\qed

\section{Conclusion and open problems}

The~aim of the paper was to study the~notion of a~P-property for two ideals induced by partitions of natural numbers, see Definition~\ref{definiciaPJ}. We chose standard critical ideals appearing in the~literature~\cite{BrFl17,BrFaVe,Hr,hrusakborel}, namely $\fin\times\emptyset$,  $\sel$, $\ED$, $\emptyset\times\fin$, $\finfin$. In the~case of ideals induced by the~same partition, we easily answered the~question which of such pairs satisfy the~P-property for two ideals. The~result is presented in~Table~\ref{omxom}. However, once considering pairs of ideals induced by potentially different partitions, the~answer appeared to be much harder. In such cases we seek a~combinatorial property of two partitions guaranteeing the~corresponding induced ideals possess the~P-property for two ideals. 
\par
Considering the~partition~$\A$, the~first ideal in the~list induced by~$\A$ is $\ofinp{\A}$. However, the~ideal is a~P-ideal, and hence a~$\pid{\J}$-ideal for an arbitrary~$\J$. We have found easy conditions characterizing the~relation of partitions~$\A,\B$ such that $\finop{\A}$ and $\finfinp{\A}$ are $\pid{\J}$-ideals for $\J$ induced by~$\B$, see Table~\ref{omxom1}. 
\par
For the~selector ideal $\selp{\A}$ and the~eventually different ideal $\EDp{\A}$, we were able to reach our goal just partially. Namely, we were able to find full characterizations in four cases displayed in Table~\ref{overview}, see Theorems~\ref{ref1}, \ref{veze}, and~\ref{ed_ofin}. We were still able to find a~useful necessary condition on partitions~$\A,\B$ such that $\selp{\A}$ is a~$\pid{\J}$-ideal for $\J$ induced by~$\B$, see Proposition~\ref{sufficient}. By Theorems~\ref{ref1}, \ref{ed_ofin}, our necessary conditions are also sufficient ones in two cases. However, as our Theorem~\ref{tazkaThm}, and Proposition~\ref{selnotp-ex_towers} show, it is not the~case for all considered ideals. 

 \begin{center}
 \begin{table}[H]
\centering
\begin{tabular}{|c|c|c|c|c|c|c|}
\hline
 & $\mathrm{P}$ &  $\mathrm{P}(\finop{\B})$ & $\mathrm{P}(\selp{\B})$  & $\mathrm{P}(\EDp{\B})$ & $\mathrm{P}(\ofinp{\B})$  & $\mathrm{P}(\finfinp{\B})$  \\
 \hline\hline
  $\ofinp{\A}$& \cmark & \cmark &  \cmark&\cmark  &\cmark &\cmark \\
  \hline
  $\finop{\A}$ & \xmark & $\rsub{}{}$ & $\rsub{}{}$ &$\rsub{}{}$ &$\rsub{}{}$ &$\rsub{}{}$ \\
   \hline
 $\finfinp{\A}$ & \xmark & $\uparrow$ & $\uparrow$ & $\uparrow$ & $\uparrow$ & $\uparrow$ \\
  \hline
   $\selp{\A}$ & \xmark & $\rsub{}{}$  & ?  & ? & $\rsub{}{}$  & ?  \\
  \hline
  $\EDp{\A}$ & \xmark  & $\perp,\rsub{}{},\uparrow$  & ?  & ? & $\perp,\rsub{}{}$   & ?    \\
    \hline
\end{tabular}
   \caption{Characterizations of $\pid{\J}$-relationships between all combinations of considered ideals.}
   \label{overview}
 \end{table}
\end{center}
The~upward pointing arrow~$\uparrow$ in the table means that the relationship is characterized by the relationship provided above the corresponding cell, e.g., $\finfinp{\A}$ is a $\pid{\finop{\B}}$ if and only if $\rsub{\finop{\A}}{\finop{\B}}$.

Thus, we leave the~question when are the~ideals $\selp{\A}$ and $\EDp{\A}$ a~$\pid{\J}$-ideal for $\J$ induced by~$\B$ for cases displayed in Table~\ref{overview} open. Among them, we believe the~following case is crucial:
\begin{quest}
    Is there a~reasonable combinatorial characterization of partitions~$\A,\B$ such that $\selp{\A}$ is a~$\pid{\selp{B}}$-ideal?
\end{quest}
Finally, there is one more ideal induced by a~partition of natural numbers and appearing very often in the~literature, namely Kat\v etov power ideal $\fin^\alpha$~\cite{BFMS,kat72}. We have not yet touched it, and therefore the~same question about this ideal seems to be fully open.  
\par


\vspace{0.5cm}

\input{biblio.tex}
\end{document}

%% file: preambula.tex
\usepackage{graphicx}
\usepackage{amsmath, amsthm,amssymb}
\usepackage{color,amscd}
\usepackage{tikz}
\usetikzlibrary{matrix}
\usepackage{float}
\usetikzlibrary{arrows, intersections}
\pgfdeclarelayer{bg}
\pgfdeclarelayer{fg}
\pgfdeclarelayer{crossing over}
\pgfsetlayers{main,bg,crossing over,fg}
\usepackage{mathtools}

\definecolor{brown}{rgb}{0.43, 0.21, 0.1}
\definecolor{green}{cmyk}{0.64,0,0.95,0.40}

\renewcommand{\restriction}{\mathord{\upharpoonright}}

\def\ml{l\kern-0.035cm\char39\kern-0.03cm}

\numberwithin{equation}{section}

\newtheorem{theorem}{Theorem}[section]
\newtheorem*{theorem*}{Theorem}
\newtheorem{lema}[theorem]{Lemma}
\newtheorem{corol}[theorem]{Corollary}
\newtheorem{prop}[theorem]{Proposition}

\theoremstyle{definition}
\newtheorem{defi}[theorem]{Definition}
\newtheorem{quest}{Question}
\newenvironment{proofofclaim}{\noindent{\textit{Proof of the claim.} }}{\hfill $\Box$ \\}

\newcommand{\dom}[1]{{\mathrm{dom}(#1)}}
\newcommand{\proj}[2]{{\mathrm{proj_{#1}}\left[#2\right]}}
\newcommand{\baire}{{}^{\omega}}

\makeatletter
\newcommand*{\rom}[1]{\expandafter\@slowromancap\romannumeral #1@}
\makeatother

\newcommand{\fin}{\text{\rm Fin}}

\newcommand{\seqn}[2]{\langle #1\colon #2\in \omega \rangle}
\newcommand{\gseqn}[2]{\langle #1\colon #2 \rangle}

\newcommand{\set}[2]{\{#1:\ #2\}}
\newcommand{\rsub}[2]{#1\subseteq^\upharpoonright #2}
\newcommand{\tow}[2]{(#1,#2)\text{-}\mathrm{TMF}}
\newcommand{\tower}[2]{(#1,#2)\text{-}\mathrm{TF}}

\newcommand{\chain}[3]{(#1,[#2],[#3])\text{-}\mathrm{chain}}
\newcommand{\I}{{\mathcal I}}
\newcommand{\J}{{\mathcal J}}
\newcommand{\K}{{\mathcal K}}

\newcommand{\V}{{\mathcal V}}
\newcommand{\A}{{\mathcal A}}

\newcommand{\B}{{\mathcal B}}
\newcommand{\C}{{\mathcal C}}
\newcommand{\E}{{\mathcal E}}
\newcommand{\W}{{\mathcal W}}
\newcommand{\D}{{\mathcal D}}
\newcommand{\F}{{\mathcal F}}
\newcommand{\T}{{\mathcal T}}
\newcommand{\G}{{\mathcal G}}

\newcommand{\PP}{{\mathcal P}}

\newcommand{\Ss}{{\mathcal S}}

\DeclarePairedDelimiter\ev{\langle}{\rangle}%

\DeclarePairedDelimiter\floor{\lfloor}{\rfloor}

\newcommand{\gen}[1]{\left\langle #1 \right\rangle}

\newcommand{\abs}[1]{\left|#1\right|}

\newcommand{\ED}{{\mathcal{ED}}}

\newcommand{\finfin}{\fin\times\fin}

\newcommand{\fino}{\fin\times\emptyset}
\newcommand{\ofin}{\emptyset\times\fin}
\newcommand{\sel}{{\mathcal Sel}}

\newcommand{\finfinp}[1]{\langle#1\rangle_{\fin\times\fin}}
\newcommand{\finop}[1]{\langle#1\rangle_{\fin\times\emptyset}}
\newcommand{\ofinp}[1]{\langle#1\rangle_{\emptyset\times\fin}}
\newcommand{\EDp}[1]{\langle#1\rangle_{\mathcal{ED}}}
\newcommand{\selp}[1]{\langle#1\rangle_{\mathcal Sel}}

\newcommand{\pid}[1]{{\mathrm P}(#1)}

\input cyracc.def

\usepackage{pifont}
\newcommand{\cmark}{\ding{51}}%
\newcommand{\xmark}{\ding{55}}%
\usepackage{multirow}
\newtheorem{claim}{Claim}
\newtheorem{innercustomthm}{Claim}
\newenvironment{clm}[1]
  {\renewcommand\theinnercustomthm{#1}\innercustomthm}
  {\endinnercustomthm}